\documentclass[a4paper, 11pt,reqno]{amsart}
\usepackage{amsfonts}
\usepackage{amsmath,amsthm}
\usepackage{amssymb}
\usepackage{amscd}
\usepackage{enumerate}
\usepackage{hyperref}

\hypersetup{
    colorlinks=true,           linkcolor=red,              citecolor=blue,            filecolor=magenta,          urlcolor=cyan           }
\newtheorem{theorem}{Theorem}[section]

\newtheorem{definition}[theorem]{Definition}

\newtheorem{lemma}[theorem]{Lemma}

\newtheorem{proposition}[theorem]{Proposition}
\newtheorem{remark}[theorem]{Remark}

\def\nc{\newcommand}

\def\be{\beta}

\nc\pa{\partial}
\nc\CC{\mathbb{C}}
\nc\RR{\mathbb{R}}
\nc\QQ{\mathbb{Q}}
\nc\ZZ{\mathbb{Z}}
\nc\NN{\mathbb{N}}

\def\ba{\begin{align}}
\def\bad{\begin{aligned}}
\def\be{\begin{equation}}
\def\ea{\end{align}}
\def\ead{\end{aligned}}
\def\ee{\end{equation}}

\input{tcilatex}

\begin{document}
\title{Global existence and blow-up of solutions for a system of fractional
wave equations}
\author{Ahmad Bashir}
\thanks{E-mail address: bashirahmad\_qau@yahoo.com}
\author{Mohamed Berbiche}
\thanks{E-mail address: mohamed.berbiche@univ-biskra.dz, berbichemed@yahoo.fr%
}
\author{Ahmed Alsaedi}
\thanks{E-mail address: aalsaedi@hotmail.com }
\author{Mokhtar Kirane}
\thanks{E-mail address: mokhtar.kirane@univ-lr.fr, }
\date{}

\begin{abstract}
We investigate the Cauchy problem for a $2\times 2$-system of weakly coupled
semi-linear fractional wave equations with polynomial nonlinearities posed
in $\mathbb{R}^{+}\times \mathbb{R}^{N}$. Under appropriate conditions on
the exponents and the fractional orders of the time derivatives, it is shown
that there exists a threshold value of the dimension $N$, for which, small
data-global solutions as well as finite time blowing-up solutions exist.
Furthermore, we investigate the $L^{\infty }$-decay estimates of global
solutions.
\end{abstract}

\maketitle

\section{Introduction}

We consider the following Cauchy problem 
\begin{equation}
\left\{ 
\begin{array}{c}
^{C}D_{0|t}^{\gamma _{1}}u-\Delta u=f(v(t,.)),\;\;t>0,x\in \mathbb{R}^{n}, \\%
[5pt] 
^{C}D_{0|t}^{\gamma _{2}}v-\Delta v=g(u(t,.)),\;\;t>0,x\in \mathbb{R}^{n},%
\end{array}%
\right.  \label{sys1}
\end{equation}%
subject to the initial conditions 
\begin{equation}
\left\{ 
\begin{array}{l}
u(0,x)=u_{0}(x),u_{t}(0,x)=u_{1}\left( x\right) ,\;\;x\in \mathbb{R}^{n}, \\%
[5pt] 
v(0,x)=v_{0}(x),\;v_{t}(0,x)=v_{1}(x),\;\;\;x\in \mathbb{R}^{n},%
\end{array}%
\right.  \label{initdat}
\end{equation}%
where $^{C}D_{0|t}^{\alpha }u$ denotes the Caputo derivative, defined for a
function $u$ of class $C^{2}$, as 
\begin{equation*}
\left( ^{C}D_{0|t}^{\alpha }u\right) \left( t\right) :=\frac{1}{\Gamma
\left( 2-\alpha \right) }\displaystyle\int_{0}^{t}\frac{u_{tt}\left(
s,.\right) }{\left( t-s\right) ^{\alpha -1}}\,ds,\text{ }1<\alpha <2\;\text{%
(see, e.g. \cite{SKM}),}
\end{equation*}%
$\Delta $ is the Laplacian, $f(v)=\pm \left\vert v\right\vert ^{p-1}v$ or $%
\pm \left\vert v\right\vert ^{p},$ $g(u)=\pm \left\vert u\right\vert ^{q-1}u$
or $\pm \left\vert u\right\vert ^{q}$, $p,q\geq 1$, and $u_{0}$, $v_{0}$, $%
u_{1}$, $v_{1}$ are given initial data.\newline

Observe that system (\ref{sys1}) interpolates reaction-diffusion system ($%
\gamma _{1}=\gamma _{2}=1$) and hyperbolic system ($\gamma _{1}=\gamma
_{2}=2 $).\newline

Before we present our results and comment on them, let us dwell on some
related existing results.\newline

Escobedo and Herrero \cite{Escobedo} studied the global existence and
blowing-up solutions of the system 
\begin{equation}  \label{EscoHerr}
\left\{ 
\begin{array}{c}
u_{t}-\Delta u=v^{p}, \quad t>0, x\in \mathbb{R}^{N}, \bigskip \\ 
v_{t}-\Delta v=u^{q}, \quad t>0, x\in \mathbb{R}^{N}.%
\end{array}
\right.
\end{equation}
In particular, for 
\begin{equation*}
pq >1, \quad \frac{N}{2} \leq \frac{\max\lbrace p, q \rbrace +1}{pq - 1},
\end{equation*}
they have shown that every nontrivial solution of \eqref{EscoHerr} blows-up
in a finite time $T^{*}=T^{*}(u, v)$, and 
\begin{equation*}
\limsup_{t \rightarrow T^{*}}\Vert u(t)\Vert_{\infty}= \limsup_{t
\rightarrow T^{*}}\Vert v(t)\Vert_{\infty}= + \infty.
\end{equation*}
Some related results concerning global existence or blowing-up solutions can
be found in \cite{FilaUda}, \cite{Samarski}, \cite{PangSunWang}, \cite{Pozio}%
, \cite{Redlinger}, etc. In particular, see the review papers \cite%
{Denglevine}, \cite{Bandle} and the authoritative paper \cite{Pohozaev}.%
\newline

Blowing-up solutions and global solutions for time-fractional differential
systems have been studied, for example, in \cite{KLT}, \cite{ZhangQuan}, 
\cite{HakemBer}, \cite{Zacher}, \cite{EidelKoch}, \cite{AlmeidaEJDE}, \cite%
{AlmeidaJMAA}, \cite{AlmeidaDIE}, \cite{YF1}, \cite{HirataMiao}, \cite%
{ZhangQuan}.\newline

Concerning the system of wave equations 
\begin{equation}
\left\{ 
\begin{array}{c}
u_{tt}-\Delta u=|v|^{p},\quad 0<t<T,x\in \mathbb{R}^{N}, \\ 
v_{tt}-\Delta v=|u|^{q},\quad 0<t<T,x\in \mathbb{R}^{N},%
\end{array}%
\right.  \label{SystWave}
\end{equation}%
subject to initial data 
\begin{equation}
\left\{ 
\begin{array}{c}
u(0,x)=f(x),\quad u_{t}(0,x)=g(x),\;\;x\in \mathbb{R}^{N}, \\ 
v(0,x)=h(x),\quad v_{t}(0,x)=k(x),\;\;x\in \mathbb{R}^{N},%
\end{array}%
\right.  \label{incd}
\end{equation}%
where $f,g,h,k\in C_{0}^{\infty }(\mathbb{R}^{N})$, we may mention the works 
\cite{KDeng}, \cite{Deng} and \cite{DelSantoMitidieri}. For $N=3$ in \cite%
{DelSantoMitidieri}, the following optimal results were obtained:\newline

$\triangleright$ If $p, q>1$ and 
\begin{equation*}
\max \left\lbrace \frac{p+2+q^{-1}}{pq-1}, \frac{q+2+p^{-1}}{pq-1}
\right\rbrace >1,
\end{equation*}
then the classical solution to \eqref{SystWave}-\eqref{incd} blows-up in a
finite time.\newline

$\triangleright$ If $p, q>1$ and 
\begin{equation*}
\max \left\lbrace \frac{p+2+q^{-1}}{pq-1}, \frac{q+2+p^{-1}}{pq-1}
\right\rbrace <1,
\end{equation*}
then there exists a global classical solution to \eqref{SystWave}-%
\eqref{incd} for sufficiently ``small'' initial data.\newline

Our interest in \eqref{sys1} stems from the fact that it interpolates
different situations; for example, reaction-diffusion systems with
fractional derivatives can model chemical reactions taking place in porous
media. In this case, fractional (nonlocal) terms with order in $(0, 1)$
account for the anomalous diffusion \cite{Magin}, \cite{Metzler}.
Experimental results show that several complex systems have a non-local
dynamics. \newline

On the other hand, equations/systems of fractional differential equations
with order in $(1, 2)$ have been studied in \cite{Mainardi}, \cite{Tarasov}, 
\cite{ChenHolm}, etc. Examples include mechanical, acoustical, biological
phenomena, marine sediments, etc. \cite{Straka}, \cite{Kappeler}. \newline

In the present paper, we consider the problem \eqref{sys1}-\eqref{initdat}
and present conditions, relating the space dimension $N$ with the parameters 
$\gamma _{1}$, $\gamma _{2}$, $p$, and $q$, for which the solution of %
\eqref{sys1}-\eqref{initdat} exists globally in time and satisfies $%
L^{\infty }$-decay estimates. We also investigate blowing-up in finite time
solutions with initial data having positive average. Our study of global
existence employs the mild formulation of the solution via Mittag-Leffler's
function, while we use the test function approach due to Mitidieri and
Pohozaev \cite{Pohozaev} for the case of blowing-up solutions. The test
function approach has been used by several authors, for instance, see \cite%
{KLT},\cite{KirFin},\cite{PangSunWang}, \cite{ZhangQuan},\cite{Bermdemp},%
\cite{Bermdemp}). To the best of our knowledge, there do not exist global
existence and large time behavior results for the time-fractional diffusion
system with two different fractional powers. Thus our results are new and
contribute significantly to the existing literature on the topic. \newline

The rest of this paper is organized as follows. In next section, we present
some preliminary Lemmas, basic facts and useful tools such as time
fractional derivative, $L^{p}$-$L^{q}$-estimates of the fundamental solution
of the linear time fractional wave equation. section 3 contains the main
results of the paper. Finally, section 4 and section 5 are devoted to the
proof of small data global existence and blow-up in finite time of the
solution of problem (\ref{sys1})-(\ref{initdat}).\newline

In the sequel, $C$ will be a positive constant which may have different
values from line to line. The space $L^{p}\left( \mathbb{R}^{N}\right) $ $%
\left( 1\leq p<\infty \right) $ will be equipped with the norm: 
\begin{equation*}
\left\Vert u\right\Vert _{L^{p}\left( \mathbb{R}^{N}\right) }^{p}=%
\displaystyle\int_{\mathbb{R}^{N}}\left\vert u\left( t,x\right)
\right\vert^{p}dx.
\end{equation*}

\section{Preliminaries}

The Riemann-Liouville fractional integral of order $0<\alpha <1$ of $f(t)
\in L^{1}\left( 0,T\right)$ is defined as 
\begin{equation*}
\left( J_{0|t}^{\alpha }f\right) (t) =\frac{1}{\Gamma (\alpha ) }
\int_{0}^{t}\left( t-\tau \right) ^{\alpha -1}f( \tau) \, d\tau,
\end{equation*}
where $\Gamma $ stands for the usual Euler gamma function.\newline

The left-sided Riemann-Liouville derivative $D_{0|t}^{\alpha }f$ (see \cite%
{SKM}), for $f\in C^{m-1}(0,T),$ of order $\alpha $ is defined as follows: 
\begin{equation*}
\left( D_{0|t}^{\alpha }f\right) (t)=\frac{d^{m}}{dt^{m}}\left(
J_{0|t}^{m-\alpha }f\right) (t),\text{ }t>0,\text{ }m-1<\alpha <m,\text{ }%
m\in \mathbb{N}.
\end{equation*}%
The Caputo fractional derivative of a function $f\in C^{m}(0,T)$ is defined
as 
\begin{equation*}
\left( ^{C}D_{0|t}^{\alpha }f\right) (t)=J_{0|t}^{m-\alpha }f^{(m)}(t),\text{
}t>0,\text{ }m-1<\alpha <m,\text{ }m\in \mathbb{N}.
\end{equation*}%
For $0<\alpha <1$ and $f$ of class $C^{1}$, we have 
\begin{equation*}
\left( D_{0|t}^{\alpha }f\right) \left( t\right) =\frac{1}{\Gamma \left(
1-\alpha \right) }\left[ \frac{f\left( 0\right) }{t^{\alpha }}+\int_{0}^{t}%
\frac{f^{\prime }\left( \sigma \right) }{\left( t-\sigma \right) ^{\alpha }}%
d\sigma \right] ,
\end{equation*}%
and 
\begin{equation}
\left( D_{t|T}^{\alpha }f\right) \left( t\right) =\frac{1}{\Gamma \left(
1-\alpha \right) }\left[ \frac{f\left( T\right) }{\left( T-t\right) ^{\alpha
}}-\int_{t}^{T}\frac{f^{\prime }\left( \sigma \right) }{\left( \sigma
-t\right) ^{\alpha }}d\sigma \right] .  \label{eq:3}
\end{equation}%
The Caputo derivative is related to the Riemann-Liouville derivative for $%
f\in AC\left[ 0,T\right] $ (the space of absolutely continuous functions
defined on $\left[ 0,T\right] $) by 
\begin{equation*}
\left( ^{C}D_{0|t}^{\alpha }f\right) (t)=D_{0|t}^{\alpha }\left(
f(t)-f(0)\right) .
\end{equation*}%
Assume that $0<\alpha <1,f\in C^{1}([a,b])$ and $g\in C(a,b)$. Then the
formula of integration by parts is 
\begin{equation*}
\int_{a}^{b}f(t)(D_{0|t}^{\alpha
}g)(t)\,dt=\int_{a}^{b}g(t)(^{C}D_{t|T}^{\alpha }f)(t)\,dt+f(a)(I_{a\mid
t}^{1-\alpha }g)(t)\Big\vert_{t=a}^{t=b}.
\end{equation*}%
The Mittag-Leffler function is defined (see \cite{SKM}) by : 
\begin{equation*}
E_{\alpha ,\beta }\left( z\right) =\sum_{k=0}^{\infty }\frac{z^{k}}{\Gamma
\left( \alpha k+\beta \right) }\text{, }\alpha \text{, }\beta \in \mathbb{C}%
\text{, }\Re \left( \alpha \right) >0\text{, }z\in \mathbb{C};
\end{equation*}%
its Riemann-Liouville fractional integral satisfies 
\begin{equation*}
J_{0|t}^{1-\alpha }\left( t^{\alpha -1}E_{\alpha ,\alpha }\left( \lambda
t^{\alpha }\right) \right) =E_{\alpha ,1}\left( \lambda t^{\alpha }\right) 
\text{ for }\lambda \in \mathbb{C},0<\alpha <1.
\end{equation*}%
For later use, let 
\begin{equation*}
\varphi \left( t\right) =\left( 1-\frac{t}{T}\right) _{+}^{l}\text{, }l\geq
2;
\end{equation*}%
then 
\begin{equation*}
^{C}D_{t|T}^{\alpha }\varphi \left( t\right) =\frac{\Gamma \left( l+1\right) 
}{\Gamma \left( l+1-\alpha \right) }T^{-\alpha }\left( 1-\frac{t}{T}\right)
_{+}^{l-\alpha }\text{, }t\leq T,
\end{equation*}%
(see for example \cite{KLT}).

\subsection{Linear estimates}

In this section, we present fundamental estimates which will be used to
prove Theorem \ref{GELT}.\newline

For $1<\alpha <2$, we define the operators $\tilde{E}_{\alpha ,1}\left(
t,x\right) $ and $\tilde{E}_{\alpha ,\alpha }\left( t,x\right) $ as follows. 
\begin{equation}
\tilde{E}_{\alpha ,1}(t,x)=\left( 2\pi \right) ^{-N/2}\mathcal{F}^{-1}\left(
E_{\alpha ,1}(-4\pi ^{2}t^{\alpha }|\xi |^{2})\right) ,\;x\in \mathbb{R}^{N},%
\text{ }t>0,  \label{Mit-Lef}
\end{equation}
\begin{equation}
\tilde{E}_{\alpha ,2}(t,x)=\left( 2\pi \right) ^{-N/2}\mathcal{F}^{-1}\left(
E_{\alpha ,2}(-4\pi ^{2}t^{\alpha }|\xi |^{2})\right) ,\;x\in \mathbb{R}^{N},%
\text{ }t>0,  \label{Mitleft}
\end{equation}
\begin{equation}
\tilde{E}_{\alpha ,\alpha }\left( t,x\right) =\left( 2\pi \right) ^{-N/2}%
\mathcal{F}^{-1}\left( E_{\alpha ,\alpha }\left( -\left\vert \xi \right\vert
^{2}t^{\alpha }\right) \right) ,\text{ }t\geq 0,\text{ }x\in \mathbb{R}^{N},%
\text{ }t>0\text{.}  \label{salpha}
\end{equation}
Consider the following linear inhomogeneous time fractional equation with
initial data: 
\begin{equation}
\left\{ 
\begin{array}{l}
^{C}D_{0|t}^{\alpha }u-\Delta u=f(t,x),\text{ }t>0,\text{ }x\in \mathbb{R}%
^{N},\text{ }1<\alpha <2, \\ 
u(0,x)=u_{0}\left( x\right) ,\text{ }u_{t}(0,x)=u_{1}(x),\text{ }x\in 
\mathbb{R}^{N}\text{.}%
\end{array}
\right.  \label{hlfde}
\end{equation}

If $u_{0}\in \mathcal{S}(\mathbb{R}^{N})$ (the Schwartz space), $u_{1}\in 
\mathcal{S}(\mathbb{R}^{N})$ and $f\in L^{1}\left( \left( 0,+\infty \right) ,%
\mathcal{S}(\mathbb{R}^{N})\right) $, then by \cite{HirataMiao} (see also 
\cite{AlmeidaDIE}) problem (\ref{hlfde}) admits a solution $u\in C^{\alpha
}\left( \left[ 0,+\infty \right) ;\mathcal{S}(\mathbb{R}^{N})\right)$, which
satisfies 
\begin{equation*}
u(t,x)=\tilde{E}_{\alpha ,1}(t,x)u_{0}(x)+t\tilde{E}_{\alpha
,2}(t,x)u_{1}(x)+\displaystyle\int_{0}^{t}\left( t-s\right) ^{\alpha -1}%
\tilde{E}_{\alpha ,\alpha }\left( t-s\right) \ast f\left( s,x\right) ds.
\end{equation*}

The following lemmas contain the so called smoothing effect of the
Mittag-Leffler operators family $\left\{ \tilde{E}_{\alpha ,1}(t)\right\}
_{t\geq 0}$ and $\left\{ \tilde{E}_{\alpha ,\alpha }(t)\right\} _{t\geq 0}$%
in Lebesgue spaces and play an important role in obtaining the first result
of this paper; they appear in \cite[Lemma 5.1]{HirataMiao}, \cite[Lemma 5.1]%
{AlmeidaEJDE}. Their proofs are based on the Fourier multiplier theorem
combined with a scaling argument (see \cite[Lemma 3.1-(i)]{AlmeidaJMAA}, 
\cite[Proposition 4.2 and Proposition 4.3]{AlmeidaEJDE}).


\begin{lemma}[{\protect\cite[Lemma 5.1]{AlmeidaEJDE}}]
\label{galpha} Let $1<p_{1}\leq p_{2}<\infty $,$\ 1<\alpha <2$ and $\lambda =%
\frac{N}{p_{1}}-\frac{N}{p_{2}}$. Then there is a constant $C>0$ such that 
\begin{align}
\Vert \tilde{E}_{\alpha ,1}(t)f\Vert _{L^{p_{2}}}& \leq Ct^{-\frac{\alpha }{2%
}\lambda }\Vert f\Vert _{L^{p_{1}}},\;\;\;\;\text{ if }\;\lambda <2\text{,}
\label{item-i} \\
\Vert t\tilde{E}_{\alpha ,2}(t)f\Vert _{L^{p_{2}}}& \leq Ct^{1-\frac{\alpha 
}{2}\lambda }\,\Vert f\Vert _{L^{p_{1}}},\;\;\text{ if }\;\frac{2}{\alpha }%
<\lambda <2,  \label{item-ii} \\
\Vert t\tilde{E}_{\alpha ,2}(t)f\Vert _{L^{p_{2}}}& \leq Ct^{-\frac{\alpha }{%
2}\lambda }\,\Vert f\Vert _{\mathcal{\dot{H}}_{p_{1}}^{-\frac{2}{\alpha }}%
\text{ }}\text{,}\;\text{ if }\;\frac{2}{\alpha }<\lambda <2,
\label{item-iii} \\
\Vert \tilde{E}_{\alpha ,\alpha }(t)\ast f\Vert _{L^{p_{2}}}& \leq Ct^{-%
\frac{\alpha }{2}\lambda }\Vert f\Vert _{L^{p_{1}}},\;\;\;\;\text{ if }%
\;\left( 2-\frac{2}{\alpha }\right) <\lambda <2,  \label{item-iv}
\end{align}%
for all $f\in \mathcal{S}^{\prime }(\mathbb{R}^{N})$, where $\mathcal{\dot{H}%
}_{p_{1}}^{-\frac{2}{\alpha }}$ is the homogeneous Sobolev spaces of
negative order $-\frac{2}{\alpha }$.
\end{lemma}

\begin{lemma}
\label{Linfty}The family of operators $\left\{ \tilde{E}_{\alpha ,1}\left(
t\right) \right\} _{t>0},$ $\left\{ \tilde{E}_{\alpha ,1}\left( t\right)
\right\} _{t>0}$ and $\left\{ \tilde{E}_{\alpha ,\alpha }\left( t\right)
\right\} _{t>0}$ enjoy the following $L^{p_{1}}-L^{p_{1}}$ estimates
property:

\begin{itemize}
\item[(i)] \textit{If }$h\in L^{p_{1}}\left( \mathbb{R}^{N}\right) $ ($1\leq
p_{1}\leq +\infty )$,\textit{\ then }$\tilde{E}_{\alpha ,\beta }\left(
t\right) h\in L^{p_{1}}\left( \mathbb{R}^{N}\right) $\textit{\ and}%
\begin{equation*}
\left\Vert \tilde{E}_{\alpha ,\beta }\left( t\right) h\right\Vert
_{L^{p_{1}}\left( \mathbb{R}^{N}\right) }\leq C\left\Vert h\right\Vert
_{L^{p_{1}}\left( \mathbb{R}^{N}\right) },\text{ }t>0\text{, for}\mathit{\ }%
\beta =1,2,\alpha ,
\end{equation*}%
for some positive constant $C>0$.

\item[(ii)] Let $p_{1}>N/2$. If $h\in L^{p_{1}}\left( \mathbb{R}^{N}\right) $%
, then $\tilde{E}_{\alpha ,\beta }\left( t\right) h\in L^{\infty }\left( 
\mathbb{R}^{N}\right) $ and we have%
\begin{equation*}
\left\Vert \tilde{E}_{\alpha ,\beta }\left( t\right) h\right\Vert
_{L^{\infty }\left( \mathbb{R}^{N}\right) }\leq Ct^{-\frac{\alpha }{2}\frac{N%
}{p_{1}}}\left\Vert h\right\Vert _{L^{p_{1}}\left( \mathbb{R}^{N}\right) }%
\text{, }t>0\text{, for}\mathit{\ }\beta =1,2,\alpha .
\end{equation*}
\end{itemize}
\end{lemma}

\begin{proof}
We use the following pointwise estimates that are shown in \cite[Theorem 5.1]%
{Kim}: 
\begin{equation*}
\left\vert \tilde{E}_{\alpha ,\alpha }\left( t,x\right) \right\vert \leq
\left\vert x\right\vert ^{-N}\exp \left\{ -c\left( t^{-\alpha }\left\vert
x\right\vert ^{2}\right) ^{\frac{1}{2-\alpha }}\right\} ,\text{ \ \ if }%
R:=\left\vert x\right\vert ^{2}t^{-\alpha }\geq 1,
\end{equation*}%
and if $R:=\left\vert x\right\vert ^{2}t^{-\alpha }<1$, then we have%
\begin{equation*}
\left\vert \tilde{E}_{\alpha ,\alpha }\left( t,x\right) \right\vert \leq
\left\{ 
\begin{array}{l}
t^{-\frac{\alpha N}{2}},\quad \quad \quad \quad \quad \quad \quad \quad
\quad N<2,, \\ 
t^{-\alpha }\left\vert x\right\vert ^{-N+2}\left( 1+\left\vert \ln \left(
\left\vert x\right\vert ^{2}t^{-\alpha }\right) \right\vert \right) ,\quad
N=2,\text{ } \\ 
\left\vert x\right\vert ^{-N+2}t^{-\alpha },\quad \quad \quad \quad \quad
\quad \,N>2.%
\end{array}%
\right.
\end{equation*}%
Concerning the operator $t\tilde{E}_{\alpha ,2}(t)$, we have the pointwise
estimates 
\begin{equation*}
\left\vert t\tilde{E}_{\alpha ,2}(t)\right\vert \leq C\left\vert
x\right\vert ^{-N}t\exp \left\{ -c\left( t^{-\alpha }\left\vert x\right\vert
^{2}\right) ^{\frac{1}{2-\alpha }}\right\} ,\text{ \ \ if }R:=\left\vert
x\right\vert ^{2}t^{-\alpha }\geq 1,
\end{equation*}%
and if $R:=\left\vert x\right\vert ^{2}t^{-\alpha }<1$, then 
\begin{equation*}
\left\vert t\tilde{E}_{\alpha ,2}\left( t,x\right) \right\vert \leq \left\{ 
\begin{array}{l}
t^{1-\frac{\alpha N}{2}},\quad \quad \quad \quad \quad \quad \quad \quad
\quad N<2,, \\ 
\left\vert x\right\vert ^{-N+2}t^{1-\alpha }\left( 1+\left\vert \ln \left(
\left\vert x\right\vert ^{2}t^{-\alpha }\right) \right\vert \right) ,\quad
N=2,\text{ } \\ 
\left\vert x\right\vert ^{-N+2}t^{1-\alpha },\quad \quad \quad \quad \quad
\quad \,N>2.%
\end{array}%
\right.
\end{equation*}%
Arguing as in Zacher et \textit{al \cite{Zacher}, } $\tilde{E}_{\alpha
,1}\left( t,.\right) ,$ $\tilde{E}_{\alpha ,2}\left( t,.\right) $ and $%
\tilde{E}_{\alpha ,\alpha }\left( t,.\right) $ are Lebesgue integrable.

In fact, we have 
\begin{equation*}
\int_{\mathbb{R}^{N}}\left\vert \tilde{E}_{\alpha ,\alpha }\left( t,x\right)
\right\vert dx=\int_{\left\{ R\geq 1\right\} }\left\vert \tilde{E}_{\alpha
,\alpha }\left( t,x\right) \right\vert dx+\int_{\left\{ R<1\right\}
}\left\vert \tilde{E}_{\alpha ,\alpha }\left( t,x\right) \right\vert \,dx;
\end{equation*}%
Using the first pointwise estimate, we get 
\begin{eqnarray*}
\int_{\left\{ R\geq 1\right\} }\left\vert \tilde{E}_{\alpha ,\alpha }\left(
t,x\right) \right\vert dx &\leq &\int_{\left\{ R\geq 1\right\} }\left\vert
x\right\vert ^{-N}\exp \left\{ -c\left( t^{-\alpha }\left\vert x\right\vert
^{2}\right) ^{\frac{1}{2-\alpha }}\right\} dx \\
&=&\int_{t^{\frac{\alpha }{2}}}^{+\infty }r^{-N}\exp \left\{ -c\left(
t^{-\alpha }r^{2}\right) ^{\frac{1}{2-\alpha }}\right\} r^{N-1}dr \\
&=&\int_{t^{\frac{\alpha }{2}}}^{+\infty }r^{-1}\exp \left\{ -c\left(
t^{-\alpha }r^{2}\right) ^{\frac{1}{2-\alpha }}\right\} dr\text{, set }z=t^{-%
\frac{\alpha }{2}}r \\
&=&\int_{1}^{+\infty }z^{-1}\exp \left\{ -c\left( z^{2}\right) ^{\frac{1}{%
2-\alpha }}\right\} dz\leq C\text{.}
\end{eqnarray*}%
On the other hand, if $N<2,$ we have 
\begin{equation*}
\int_{\left\{ R\leq 1\right\} }\left\vert \tilde{E}_{\alpha ,\alpha }\left(
t,x\right) \right\vert dx\leq \int_{\left\{ R\leq 1\right\} }t^{-\frac{%
\alpha N}{2}}dx=t^{-\frac{\alpha N}{2}}\int_{0}^{t^{\frac{\alpha }{2}%
}}r^{N-1}dr=t^{-\frac{\alpha N}{2}}\frac{t^{\frac{\alpha N}{2}}}{N}=C.
\end{equation*}%
For $N=2,$ we have 
\begin{eqnarray*}
\int_{\left\{ R\leq 1\right\} }\left\vert \tilde{E}_{\alpha ,\alpha }\left(
t,x\right) \right\vert dx &\leq &\int_{\left\{ R\leq 1\right\} }\left\vert
x\right\vert ^{-N+2}t^{-\alpha }\left( 1-\ln \left( \left\vert x\right\vert
^{2}t^{-\alpha }\right) \right) \,dx \\
&=&t^{-\alpha }\int_{0}^{t^{\frac{\alpha }{2}}}\left( 1+\left\vert \ln
\left( r^{2}t^{-\alpha }\right) \right\vert \right) rdr \\
&=&t^{-\alpha }t^{\frac{\alpha }{2}N}\int_{0}^{1}\left( 1-\ln \left(
z^{2}\right) \right) zdz=C.
\end{eqnarray*}%
When $N>2,$ we have 
\begin{eqnarray*}
\int_{\left\{ R\leq 1\right\} }\left\vert \tilde{E}_{\alpha ,\alpha }\left(
t,x\right) \right\vert dx &\leq &\int_{\left\{ R\leq 1\right\} }\left\vert
x\right\vert ^{-N+2}t^{-\alpha }dx \\
&=&t^{-\alpha }\int_{0}^{t^{\frac{\alpha }{2}}}r^{-N+2}r^{N-1}dr=t^{-\alpha
}\int_{0}^{t^{\frac{\alpha }{2}}}rdr,
\end{eqnarray*}%
so, 
\begin{equation*}
\int_{\left\{ R\leq 1\right\} }\left\vert \tilde{E}_{\alpha ,\alpha }\left(
t,x\right) \right\vert dx\leq 1/2\text{.}
\end{equation*}%
The first result (i) follows from Young's convolution inequality, that is, 
\begin{eqnarray*}
\left\Vert \tilde{E}_{\alpha ,\alpha }\left( t,.\right) h\right\Vert
_{L^{p_{1}}\left( \mathbb{R}^{N}\right) } &=&\left\Vert \tilde{E}_{\alpha
,\alpha }\left( t,.\right) \ast h\left( x\right) \right\Vert
_{L^{p_{1}}\left( \mathbb{R}^{N}\right) }\leq \left\Vert \tilde{E}_{\alpha
,\alpha }\left( t,.\right) \right\Vert _{L^{1}\left( \mathbb{R}^{N}\right)
}\left\Vert h\right\Vert _{L^{p_{1}}\left( \mathbb{R}^{N}\right) } \\
&\leq &C\left\Vert h\right\Vert _{L^{p_{1}}\left( \mathbb{R}^{N}\right) }.
\end{eqnarray*}%
In a s similar manner, it can be shown that the operators $\tilde{E}_{\alpha
,1}\left( t,.\right) $ and $\tilde{E}_{\alpha ,2}\left( t,.\right) $ are
bounded.

In order to show statment (ii), we need to prove that $\tilde{E}_{\alpha
,\alpha }\left( t,.\right) $, belongs to $L^{p_{2}}\left( \mathbb{R}%
^{N}\right) $ 
\begin{eqnarray*}
\int_{\left\{ R\geq 1\right\} }\left\vert \tilde{E}_{\alpha ,\alpha }\left(
t,x\right) \right\vert ^{p_{2}}dx &\leq &\int_{\left\{ R\geq 1\right\}
}\left\vert x\right\vert ^{-Np_{2}}\exp \left\{ -c\left( t^{-\alpha
}\left\vert x\right\vert ^{2}\right) ^{\frac{1}{2-\alpha }}\right\} dx \\
&=&\int_{t^{\frac{\alpha }{2}}}^{+\infty }r^{-Np_{2}}\exp \left\{ -c\left(
t^{-\alpha }r^{2}\right) ^{\frac{1}{2-\alpha }}\right\} r^{N-1}dr \\
&=&\int_{t^{\frac{\alpha }{2}}}^{+\infty }r^{-Nr+N-1}\exp \left\{ -c\left(
t^{-\alpha }r^{2}\right) ^{\frac{1}{2-\alpha }}\right\} dr\text{, set }z=t^{-%
\frac{\alpha }{2}}r \\
&=&t^{-\frac{\alpha }{2}N\left( p_{2}-1\right) }\int_{1}^{+\infty
}z^{-Np_{2}+N-1}\exp \left\{ -c\left( z^{2}\right) ^{\frac{1}{2-\alpha }%
}\right\} dz\leq Ct^{-\frac{\alpha }{2}N\left( p_{2}-1\right) }\text{.}
\end{eqnarray*}%
On the other hand, if $N=1,$ we have 
\begin{equation*}
\int_{\left\{ R\leq 1\right\} }\left\vert \tilde{E}_{\alpha ,\alpha }\left(
t,x\right) \right\vert ^{p_{2}}dx\leq \int_{\left\{ R\leq 1\right\} }t^{-%
\frac{\alpha N}{2}p_{2}}dx=t^{-\frac{\alpha N}{2}p_{2}}\int_{0}^{t^{\frac{%
\alpha }{2}}}r^{N-1}dr=\frac{1}{N}t^{-\frac{\alpha N}{2}p_{2}+\frac{\alpha }{%
2}}=Ct^{-\frac{\alpha }{2}N\left( p_{2}-1\right) }.
\end{equation*}%
For $N=2$, we have 
\begin{eqnarray*}
\int_{\left\{ R\leq 1\right\} }\left\vert \tilde{E}_{\alpha ,\alpha }\left(
t,x\right) \right\vert ^{p_{2}}dx &\leq &\int_{\left\{ R\leq 1\right\}
}t^{-\alpha p_{2}}\left( 1-\ln \left( \left\vert x\right\vert ^{2}t^{-\alpha
}\right) \right) ^{p_{2}}dx \\
&=&t^{-\alpha p_{2}}\int_{0}^{t^{\frac{\alpha }{2}}}\left( 1+\left\vert \ln
\left( r^{2}t^{-\alpha }\right) \right\vert \right) ^{p_{2}}r^{N-1}dr \\
&=&t^{-\alpha \left( p_{2}-1\right) }\int_{0}^{1}\left( 1-\ln \left(
z^{2}\right) \right) ^{p_{2}}zdz=Ct^{-\alpha \left( p_{2}-1\right) }.
\end{eqnarray*}%
When $N>2,$ we have 
\begin{eqnarray*}
\int_{\left\{ R\leq 1\right\} }\left\vert \tilde{E}_{\alpha ,\alpha }\left(
t,x\right) \right\vert ^{p_{2}}dx &\leq &\int_{\left\{ R\leq 1\right\}
}\left\vert x\right\vert ^{-\left( N-2\right) p_{2}}t^{-\alpha p_{2}}dx \\
&=&t^{-\alpha p_{2}}\int_{0}^{t^{\frac{\alpha }{2}}}r^{-\left( N-2\right)
p_{2}}r^{N-1}dr=t^{-\alpha p_{2}}\int_{0}^{t^{\frac{\alpha }{2}}}r^{-\left(
N-2\right) p_{2}+N-1}dr,
\end{eqnarray*}%
provided $N>(N-2)p_{2}$. So, 
\begin{equation*}
\int_{\left\{ R\leq 1\right\} }\left\vert \tilde{E}_{\alpha ,\alpha }\left(
t,x\right) \right\vert ^{p_{2}}dx\leq Ct^{-\alpha p_{2}-\frac{\alpha }{2}%
\left( N-2\right) p_{2}+\frac{\alpha }{2}N}=Ct^{-\frac{\alpha }{2}N\left(
p_{2}-1\right) }\text{.}
\end{equation*}%
Hence $\left\Vert \tilde{E}_{\alpha ,\alpha }\left( t,.\right) \right\Vert
_{p_{2}}\leq Ct^{-\frac{\alpha }{2}N\left( 1-\frac{1}{p_{2}}\right) }$, for $%
p_{2}<N/\left( N-2\right) $.

Now (ii) follows by Young's convolution inequality and the last estimate%
\begin{equation*}
\Vert \tilde{E}_{\alpha ,\alpha }(t)\ast f\Vert _{L^{\infty }}\leq \Vert 
\tilde{E}_{\alpha ,\alpha }(t)\Vert _{L^{p_{1}^{\prime }}}\Vert f\Vert
_{L^{p_{1}}}\leq Ct^{-\frac{\alpha }{2}\frac{N}{p_{1}}}\Vert f\Vert
_{L^{p_{1}}}\text{, for }p_{1}>\frac{N}{2}.
\end{equation*}%
Where $p_{1}^{\prime }$ is the conjugate of $p_{1}$ ($1/p_{1}+1/p_{1}^{%
\prime }=1$). Arguing in a similar way, we obtain $L^{p_{1}}-L^{\infty }$
estimates to the operators $\tilde{E}_{\alpha ,\beta }\left( t\right) $, for 
$\beta =1,2.$
\end{proof}

\begin{lemma}
\label{poldec}Let $l\geq 1,$ and let the function $f\left( t,x\right) $
satisfy 
\begin{equation*}
\left\Vert f(t,.)\right\Vert _{l}\leq C_{1},\;0\leq t\leq 1\text{, \ }%
\;\left\Vert f(t,.)\right\Vert _{l}\leq C_{2}t^{-\alpha }\text{, }\;t>0,
\end{equation*}%
for some positive constants $C_{1},C_{2}$ and $\alpha $. Then 
\begin{equation*}
\left\Vert f\left( t,.\right) \right\Vert _{l}\leq \max \left\{
C_{1},C_{2}\right\} \left( 1+t\right) ^{-\beta }\text{, }\;\text{ for all}%
\;0<\beta \leq \alpha \;\text{ and }\;t\geq 0\text{.}
\end{equation*}
\end{lemma}

\section{Main results}

In this section, we state our main results. Let us begin with the definition
of a mild solution of problem (\ref{sys1})-(\ref{initdat}).

\begin{definition}
Let $u_{0},v_{0},u_{1},v_{1}\in \mathbb{X},(\mathbb{X}:=L^{1}(\mathbb{R}%
^{N})\cap L^{\infty }\left( \mathbb{R}^{N}\right) ),1<\gamma _{1},\gamma
_{2}<2$, $f,g\in L^{1}\left( \left( 0,T\right) ,\mathcal{S}(\mathbb{R}%
^{N})\right) $ and $T>0$. We call $\left( u,v\right) \in C\left( \left[ 0,T%
\right] ;\mathbb{X}\right) \times C\left( \left[ 0,T\right] ;\mathbb{X}%
\right) $ a mild solution of system (\ref{sys1})-(\ref{initdat}) if $\left(
u,v\right) $ satisfies the following integral%
\begin{eqnarray}
u\left( t,x\right) &=&\tilde{E}_{\gamma _{1},1}(t,x)u_{0}(x)+t\tilde{E}%
_{\gamma _{1},2}(t,x)u_{1}(x)  \notag \\
&&+\displaystyle\int_{0}^{t}\left( t-\tau \right) ^{\gamma _{1}-1}\tilde{E}%
_{\gamma _{1},\gamma _{1}}\left( t-\tau ,x\right) f(v\left( \tau ,x\right)
)\,d\tau ,  \label{ms1}
\end{eqnarray}%
\begin{eqnarray}
v\left( t,x\right) &=&\tilde{E}_{\gamma _{2},1}(t,x)v_{0}(x)+t\tilde{E}%
_{\gamma _{2},2}(t,x)v_{1}(x)  \notag \\
&&+\displaystyle\int_{0}^{t}\left( t-\tau \right) ^{\gamma _{1}-1}\tilde{E}%
_{\gamma _{2},\gamma _{2}}\left( t-\tau ,x\right) g(u\left( \tau ,x\right)
)\,d\tau .  \label{ms2}
\end{eqnarray}
\end{definition}

The existence and uniqueness of a local solution of (\ref{sys1}) can be
established by using the Banach fixed point theorem and Gronwall's
inequality.

\begin{proposition}[Local existence of a mild solution]
\label{FD} Let $u_{0}, v_{0}, u_{1}, v_{1}\in \mathbb{X} $, $1<\gamma
_{1},\gamma _{2}<2$, $p$, $q\geq 1$ such that $pq>1$. Then there exist a
maximal time $T_{\max }>0$ and a unique mild solution to problem (\ref{sys1}%
)-(\ref{initdat}), such that either

\begin{itemize}
\item[(i)] $T_{\max }=\infty $ (the solution is global), or \newline

\item[(ii)] $T_{\max }<\infty $ and $\lim\limits_{t\rightarrow T_{\max
}}\left( \left\Vert u(t)\right\Vert _{\infty }+\left\Vert v(t)\right\Vert
_{\infty }\right) =\infty $ (the solution blows up in a finite time).
\end{itemize}

Moreover, for any $s_{1},s_{2}\in \left( 1,+\infty \right) ,$ $\left(
u,v\right) \in C\left( \left[ 0,T\right] ;L^{s_{1}}\left( \mathbb{R}%
^{N}\right) \times L^{s_{2}}\left( \mathbb{R}^{N}\right) \right) .$
\end{proposition}

Now, we are in a position to state the first main result of this section
concerning global existence and large time behavior of solutions of (\ref%
{sys1})-(\ref{initdat}).

\begin{theorem}[Global existence of a mild solution]
\label{GELT} Let $N\geq 2,$ $q\geq p\geq 1,$ $pq>1$, $1<\gamma _{1}\leq
\gamma _{2}<2.$ If 
\begin{equation}
\frac{N}{2}\geq \max \left\{ \frac{1}{\gamma _{1}}+\frac{q+1}{pq-1},\frac{1}{%
\gamma _{1}}+\frac{p\gamma _{2}+\gamma _{1}}{\gamma _{1}\left( pq-1\right) }%
\right\} ,\text{\ }  \label{critdimension}
\end{equation}%
the initial data satisfy 
\begin{equation*}
\left\Vert u_{0}\right\Vert _{\mathbb{X}}+\left\Vert u_{1}\right\Vert _{%
\mathbb{X}}+\left\Vert v_{0}\right\Vert _{\mathbb{X}}+\left\Vert
v_{1}\right\Vert _{\mathbb{X}}\leq \varepsilon _{0},
\end{equation*}%
for some $\varepsilon _{0}>0$, then problem (\ref{sys1})-(\ref{initdat})
admits a global mild solution and that 
\begin{eqnarray*}
u &\in &L^{\infty }\left( \left[ 0,\infty \right) ,L^{\infty }\left( \mathbb{%
R}^{N}\right) \right) \cap L^{\infty }\left( \left[ 0,\infty \right)
,L^{s_{1}}\left( \mathbb{R}^{N}\right) \right) \text{,}\bigskip \\
v &\in &L^{\infty }\left( \left[ 0,\infty \right) ,L^{\infty }\left( \mathbb{%
R}^{N}\right) \right) \cap L^{\infty }\left( \left[ 0,\infty \right)
,L^{s_{2}}\left( \mathbb{R}^{N}\right) \right) \text{,}
\end{eqnarray*}%
where $s_{1}>q$ and $s_{2}>p.$ \newline

Furthermore, for any $\delta $ satisfying $1-\frac{1+q}{(p+1)q\gamma _{2}}%
<\delta <\min \left\{ 1,\frac{N\left( pq-1\right) }{2q(p+1)}\right\} ,$ 
\begin{equation*}
\left\Vert u\left( t\right) \right\Vert _{s_{1}}\leq C\left( t+1\right) ^{-%
\frac{\left( 1-\delta \right) \left( \gamma _{1}+p\gamma _{2}\right) }{pq-1}%
},\;t\geq 0,
\end{equation*}%
\begin{equation*}
\text{ }\left\Vert v\left( t\right) \right\Vert _{s_{2}}\leq C\left(
t+1\right) ^{-\frac{\left( 1-\delta \right) \left( \gamma _{2}+q\gamma
_{1}\right) }{pq-1}},\;t\geq 0.
\end{equation*}%
If, in addition, 
\begin{equation*}
\frac{pN}{2s_{2}}<1\;\text{ and}\;\;\;\frac{qN}{2s_{1}}<1,
\end{equation*}%
or 
\begin{equation*}
N>2,\;pN/(2s_{2})<1\;\text{ and}\;\;qN/(2s_{1})\geq 1,
\end{equation*}%
or 
\begin{equation*}
N>2,\;qN/(2s_{1})\geq 1,\;pN/(2s_{2})\geq 1\;\text{ and}\;\;q\geq p>1\;\text{
with}\;\;\sqrt{\frac{\left( p+1\right) q\gamma _{1}}{\left( q+1\right) p}}%
<\gamma _{1}\leq \gamma _{2}<2,
\end{equation*}%
then, 
\begin{eqnarray*}
u,v &\in &L^{\infty }\left( \left[ 0,\infty \right) ,L^{\infty }\left( 
\mathbb{R}^{N}\right) \right) , \\
\left\Vert u\left( t\right) \right\Vert _{\infty } &\leq &C\left( t+1\right)
^{-\tilde{\sigma}},\text{ }\left\Vert v\left( t\right) \right\Vert _{\infty
}\leq C\left( t+1\right) ^{-\hat{\sigma}},\;\,\text{for all}\;t\geq 0,
\end{eqnarray*}%
for some positive constants $\tilde{\sigma}$ and $\hat{\sigma}$.\bigskip\ 
\end{theorem}

\begin{definition}[Weak solution]
\label{Weaks} Let $u_{0},v_{0}\in L_{loc}^{\infty }\left( \mathbb{R}%
^{N}\right) ,u_{1},v_{1}\in L_{loc}^{\infty }\left( \mathbb{R}^{N}\right)
,T>0$. We say that $\left( u,v\right) \in L^{q}\left( (0,T),L_{loc}^{\infty
}\left( \mathbb{R}^{N}\right) \right) \times L^{p}\left(
(0,T),L_{loc}^{\infty }\left( \mathbb{R}^{N}\right) \right) $ is a weak
solution of (\ref{sys1}) if 
\begin{equation*}
\begin{array}{l}
\displaystyle\int_{0}^{T}\displaystyle\int_{\mathbb{R}^{N}}uD_{t|T}^{\gamma
_{1}}\varphi \left( t,x\right) dxdt-\displaystyle\int_{0}^{T}\displaystyle%
\int_{\mathbb{R}^{N}}u\Delta \varphi \left( t,x\right) dxdt=\displaystyle%
\int_{\mathbb{R}^{N}}u_{0}\left( x\right) \left( D_{t|T}^{\gamma
_{1}-1}\varphi \right) \left( 0,.\right) dx \\ 
+\displaystyle\int_{0}^{T}\displaystyle\int_{\mathbb{R}^{N}}u_{1}D_{t|T}^{%
\gamma _{1}-1}\varphi \left( t,x\right) dxdt+\displaystyle\int_{0}^{T}%
\displaystyle\int_{\mathbb{R}^{N}}f(v\left( \tau ,x\right) )\varphi \left(
t,x\right) dxdt\text{,}\bigskip \\ 
\displaystyle\int_{0}^{T}\displaystyle\int_{\mathbb{R}^{N}}vD_{t|T}^{\gamma
_{2}}\varphi \left( t,x\right) dxdt-\displaystyle\int_{0}^{T}\displaystyle%
\int_{\mathbb{R}^{N}}v\Delta \varphi \left( t,x\right) dxdt=\displaystyle%
\int_{\mathbb{R}^{N}}v_{0}\left( x\right) \left( D_{t|T}^{\gamma
_{2}-1}\varphi \right) \left( 0,.\right) dx\bigskip \\ 
+\displaystyle\int_{0}^{T}\displaystyle\int_{\mathbb{R}^{N}}v_{1}D_{t|T}^{%
\gamma _{2}-1}\varphi \left( t,x\right) dxdt+\displaystyle\int_{0}^{T}%
\displaystyle\int_{\mathbb{R}^{N}}g(u\left( \tau ,x\right) )\varphi \left(
t,x\right) dxdt\text{.}%
\end{array}%
\end{equation*}%
for every function $\varphi \in C_{t,x}^{1,2}\left( [0,T]\times \mathbb{R}%
^{N}\right) $ such that $\varphi \left( T,.\right) =0$.
\end{definition}

Similar to the proof in \cite{KirFin}, we can obtain the following lemma
asserting that the mild solution is the weak solution.

\begin{lemma}
Assume that $\left( u_{0},v_{0}\right) ,\left( u_{1},v_{1}\right) \in 
\mathcal{S}\left( \mathbb{R}^{N}\right) \times \mathcal{S}\left( \mathbb{R}%
^{N}\right) $and let $\left( u,v\right) \in C^{\gamma _{1}}\left( [0,T],%
\mathcal{S}\left( \mathbb{R}^{N}\right) \right) \times C^{\gamma _{2}}\left(
[0,T],\mathcal{S}\left( \mathbb{R}^{N}\right) \right) $ be a mild solution
of (\ref{sys1})-(\ref{initdat}). Then $\left( u,v\right) $ is also a weak
solution of (\ref{sys1})-(\ref{initdat}).
\end{lemma}

\begin{proof}
As $\left( u,v\right)$ is a mild solution, we have 
\begin{equation*}
u(t,x)=\tilde{E}_{\gamma _{1},1}(t,x)u_{0}(x)+t\tilde{E}_{\gamma
_{1},2}(t,x)u_{1}(x)+\displaystyle\int_{0}^{t}\left( t-s\right) ^{\gamma
_{1}-1}\tilde{E}_{\gamma _{1},\gamma _{1}}\left( t-s\right) \ast f\left(
v(s,x)\right) ds.
\end{equation*}
Differentiating with respect to $t$ and noting that $1<\gamma _{1}<2$, we
get 
\begin{eqnarray}
u_{t}(t,x)-u_{1}(x) &=&\partial _{t}\tilde{E}_{\gamma
_{1},1}(t,x)u_{0}(x)+\partial _{t}\left( t\tilde{E}_{\gamma
_{1},2}(t,x)\right) u_{1}(x)-u_{1}(x)  \notag \\
&&+\displaystyle\int_{0}^{t}\left( t-s\right) ^{\gamma _{1}-2}\tilde{E}%
_{\gamma _{1},\gamma _{1}-1}\left( t-s\right) \ast f\left( v(s,x\right) )ds%
\text{,}  \label{deriv}
\end{eqnarray}
where we have used the following formula 
\begin{equation*}
\left( \frac{d}{dz}\right) ^{(m)}\left[ z^{\beta -1}E_{\alpha ,\beta }\left(
z^{\alpha }\right) \right] =z^{\beta -m-1}E_{\alpha ,\beta -m}\left(
z^{\alpha }\right) ,\text{ }\Re \left( \beta -m\right) >0,\text{ }m=0,1,...
\end{equation*}%
Applying $J_{0|t}^{2-\gamma _{1}}$ to both sides of (\ref{deriv}), we obtain 
\begin{equation*}
\begin{array}{c}
J_{0|t}^{2-\gamma _{1}}\left( u_{t}-u_{1}\right) =J_{0|t}^{2-\gamma
_{1}}\left( \partial _{t}\tilde{E}_{\gamma _{1},1}(t,x)\right)
u_{0}(x)+J_{0|t}^{2-\gamma _{1}}\left( \partial _{t}\left( t\tilde{E}%
_{\gamma _{1},2}(t,.)\right) u_{1}(x)-u_{1}(x)\right) \\ 
+J_{0|t}^{2-\gamma _{1}}\left( \displaystyle\int_{0}^{t}\left( t-s\right)
^{\gamma _{1}-2}\tilde{E}_{\gamma _{1},\gamma _{1}-1}\left( t-s,.\right)
\ast f\left( v(s,x\right) )\right) ds.%
\end{array}%
\end{equation*}
On the other hand, we have 
\begin{equation}
\begin{array}{l}
J_{0|t}^{2-\gamma _{1}}\left( \displaystyle\int_{0}^{t}\left( t-s\right)
^{\gamma _{1}-2}E_{\gamma _{1},\gamma _{1}-1}\left( -\left\vert \xi
\right\vert ^{2}\left( s-\tau \right) ^{\gamma _{1}}\right) \hat{f}\left(
s,\xi \right) ds\right) \\ 
=\frac{1}{\Gamma \left( 2-\gamma _{1}\right) }\displaystyle%
\int_{0}^{t}\left( t-s\right) ^{1-\gamma _{1}}\displaystyle%
\int_{0}^{s}\left( s-\tau \right) ^{\gamma _{1}-2}E_{\gamma _{1},\gamma
_{1}-1}\left( -\left\vert \xi \right\vert ^{2}\left( s-\tau \right) ^{\gamma
_{1}}\right) \hat{f}\left( \tau ,\xi \right) d\tau ds \\ 
=\sum\limits_{k=0}^{+\infty }\frac{\left( -1\right) ^{k}\left\vert \xi
\right\vert ^{2k}}{\Gamma \left( 2-\gamma _{1}\right) \Gamma \left( \gamma
_{1}k+\gamma _{1}-1\right) }\displaystyle\displaystyle\int_{0}^{t}\left(
t-s\right) ^{1-\gamma _{1}}\displaystyle\int_{0}^{s}\left( s-\tau \right)
^{\gamma _{1}-2+\gamma _{1}k}\hat{f}\left( \tau ,\xi \right) d\tau ds \\ 
=\sum\limits_{k=0}^{+\infty }\frac{\left( -1\right) ^{k}\left\vert \xi
\right\vert ^{2k}}{\Gamma \left( 2-\gamma _{1}\right) \Gamma \left( \gamma
_{1}k+\gamma _{1}-1\right) }\displaystyle\int_{0}^{t}\displaystyle\int_{\tau
}^{t}\left( t-s\right) ^{1-\gamma _{1}}\left( s-\tau \right) ^{\gamma
_{1}-2+\gamma _{1}k}ds\hat{f}\left( \tau ,\xi \right) d\tau \\ 
=\sum\limits_{k=0}^{+\infty }\frac{\left( -1\right) ^{k}\left\vert \xi
\right\vert ^{2k}}{\Gamma \left( 2-\gamma _{1}\right) \Gamma \left( \gamma
_{1}k+\gamma _{1}-1\right) }\mathbf{B}\left( 2-\gamma _{1},\gamma
_{1}k+\gamma _{1}-1\right) \displaystyle\int_{0}^{t}\left( t-s\right)
^{\gamma _{1}k}\hat{f}\left( s,\xi \right) ds \\ 
=\displaystyle\int_{0}^{t}E_{\gamma _{1},1}\left( -\left\vert \xi
\right\vert ^{2}\left( s-\tau \right) ^{\gamma _{1}}\right) \hat{f}\left(
s,\xi \right) \, ds.%
\end{array}
\label{ftmlf}
\end{equation}
Here $\mathbf{B}$ denotes to the beta function.\newline

Applying the Fourier inverse transform to both sides of (\ref{ftmlf}) yields 
\begin{equation*}
J_{0|t}^{2-\gamma _{1}}\left( \displaystyle\int_{0}^{t}\left( t-s\right)
^{\gamma _{1}-2}\tilde{E}_{\gamma _{1},\gamma _{1}-1}\left( t-s,.\right)
\ast f\left( v(s,x\right) )\right) ds
\end{equation*}
\begin{equation*}
=\displaystyle\int_{0}^{t}\tilde{E}_{\gamma _{1},1}\left( t-s,.\right) \ast
f\left( v(s,x\right) )ds\text{.}
\end{equation*}

Then, for every test function $\varphi \in C_{x,t}^{2,1}\left( \mathbb{R}%
^{N}\times \left[ 0,T\right] \right) ,$ supp$\varphi \subset \subset \mathbb{%
R}^{N}\times \left[ 0,T\right] $ and $\varphi \left( T,x\right) =0$, we have 
\begin{equation*}
\begin{array}{l}
\displaystyle\int_{\mathbb{R}^{N}}J_{0|t}^{2-\gamma _{1}}\left(
u_{t}-u_{1}\right) \varphi dx=\displaystyle\int_{\mathbb{R}%
^{N}}J_{0|t}^{2-\gamma _{1}}\left( \partial _{t}\tilde{E}_{\gamma
_{1},1}(t,x)\right) u_{0}(x)\varphi \,dx \\ 
\quad \quad \quad \quad \quad \quad +\displaystyle\int_{\mathbb{R}%
^{N}}J_{0|t}^{2-\gamma _{1}}\left( \partial _{t}\left( t\tilde{E}_{\gamma
_{1},2}(t,.)\right) u_{1}(x)-u_{1}(x)\right) \varphi \,dx \\ 
\quad \quad \quad \quad \quad \quad +\displaystyle\int_{\mathbb{R}^{N}}%
\displaystyle\int_{0}^{t}\tilde{E}_{\gamma _{1},1}\left( t-s\right) \ast
f\left( v(s,x)\right) ds\varphi \,dx.%
\end{array}%
\end{equation*}%
Setting%
\begin{equation*}
I:=\displaystyle\int_{\mathbb{R}^{N}}J_{0|t}^{2-\gamma _{1}}\left(
u_{t}-u_{1}\right) \varphi \,dx\text{,}
\end{equation*}%
we get 
\begin{equation*}
\begin{array}{l}
\;\;\frac{\partial }{\partial t}I_{{}}=\displaystyle\int_{\mathbb{R}^{N}}%
\frac{\partial }{\partial t}\left[ J_{0|t}^{2-\gamma _{1}}\left(
u_{t}-u_{1}\right) \varphi \right] dx \\ 
\quad \quad =\displaystyle\int_{\mathbb{R}^{N}}\frac{\partial }{\partial t}%
\left[ J_{0|t}^{2-\gamma _{1}}\left( \partial _{t}\tilde{E}_{\gamma
_{1},1}(t,x)\right) u_{0}(x)\varphi \right] dx \\ 
\quad \quad +\displaystyle\int_{\mathbb{R}^{N}}\frac{\partial }{\partial t}%
\left[ J_{0|t}^{2-\gamma _{1}}\left( \partial _{t}\left( t\tilde{E}_{\gamma
_{1},2}(t,.)\right) u_{1}(x)-u_{1}(x)\right) \varphi \right] dx \\ 
\quad \quad +\displaystyle\int_{\mathbb{R}^{N}}\frac{\partial }{\partial t}%
\left( \displaystyle\int_{0}^{t}\tilde{E}_{\gamma _{1},1}\left( t-s\right)
d\tau \ast f\left( s,x\right) ds\varphi \right) dx.%
\end{array}%
\end{equation*}%
On the other hand, using the relations 
\begin{equation*}
D_{0|t}^{\gamma _{1}}\tilde{E}_{\gamma _{1},1}\left( t,.\right) u_{0}\left(
x\right) =\Delta \tilde{E}_{\gamma _{1},1}\left( t,.\right) u_{0}\left(
x\right) ,
\end{equation*}%
\begin{equation*}
D_{0|t}^{\gamma _{1}}\left( t\tilde{E}_{\gamma _{1},2}(t,.)\right)
u_{1}(x)=\Delta \left( t\tilde{E}_{\gamma _{1},2}(t,.)\right) u_{1}(x)\text{,%
}
\end{equation*}%
we obtain 
\begin{eqnarray*}
&&\int_{\mathbb{R}^{N}}\frac{\partial }{\partial t}\left[ J_{0|t}^{2-\gamma
_{1}}\left( \partial _{t}\left( t\tilde{E}_{\gamma _{1},2}(t,.)\right)
u_{1}(x)-u_{1}(x)\right) \varphi \right] dx \\
&=&\int_{\mathbb{R}^{N}}D_{0|t}^{\gamma _{1}}\left( t\tilde{E}_{\gamma
_{1},2}(t,.)\right) u_{1}(x)\varphi \left( t,x\right) dx \\
&&+\int_{\mathbb{R}^{N}}J_{0|t}^{2-\gamma _{1}}\left( \partial _{t}\left( t%
\tilde{E}_{\gamma _{1},2}(t,.)\right) u_{1}(x)-u_{1}(x)\right) \varphi
_{t}\left( t,x\right) dx \\
&=&\int_{\mathbb{R}^{N}}t\tilde{E}_{\gamma _{1},2}(t,.)u_{1}(x)\Delta
\varphi \left( t,x\right) dx \\
&&+\int_{\mathbb{R}^{N}}J_{0|t}^{2-\gamma _{1}}\left( \partial _{t}\left( t%
\tilde{E}_{\gamma _{1},2}(t,.)\right) u_{1}(x)-u_{1}(x)\right) \varphi
_{t}\left( t,x\right) dx\text{,}
\end{eqnarray*}%
and%
\begin{equation*}
\int_{\mathbb{R}^{N}}\frac{\partial }{\partial t}\left[ J_{0|t}^{2-\gamma
_{1}}\left( \partial _{t}\tilde{E}_{\gamma _{1},1}(t,x)\right)
u_{0}(x)\varphi \right] dx=\int_{\mathbb{R}^{N}}\tilde{E}_{\gamma
_{1},1}(t,x)u_{0}(x)\Delta \varphi \left( t,x\right) dx
\end{equation*}%
\begin{equation*}
\quad \quad \quad \quad \quad \quad \quad \quad \quad \quad +\int_{\mathbb{R}%
^{N}}J_{0|t}^{2-\gamma _{1}}\left( \partial _{t}\tilde{E}_{\gamma
_{1},1}(t,x)\right) u_{0}(x)\varphi _{t}\left( t,x\right) dx\text{.}
\end{equation*}%
Using the Leibniz formula, we get 
\begin{equation*}
\frac{\partial }{\partial t}\displaystyle\int_{0}^{t}\tilde{E}_{\gamma
_{1},1}\left( t-s\right) \ast f\left( v(s,x)\right) ds
\end{equation*}%
\begin{equation*}
\quad \quad \quad \quad \quad \quad \quad \quad =\tilde{E}_{\gamma
_{1},1}\left( 0\right) f\left( v(t,x)\right) +\displaystyle%
\int_{0}^{t}\partial _{t}\tilde{E}_{\gamma _{1},1}\left( t-s\right) \ast
f\left( v(s,x)\right) \,ds
\end{equation*}%
\begin{equation*}
\quad \quad \quad \quad =f\left( v(t,x\right) )+\displaystyle%
\int_{0}^{t}\partial _{t}\tilde{E}_{\gamma _{1},1}\left( t-s\right) \ast
f\left( v(t,x)\right) \,ds.
\end{equation*}%
So, 
\begin{equation*}
\begin{array}{l}
\frac{\partial }{\partial t}I=\displaystyle\int_{\mathbb{R}^{N}}\tilde{E}%
_{\gamma _{1},1}(t,x)u_{0}(x)\Delta \varphi dx+\displaystyle\int_{\mathbb{R}%
^{N}}t\tilde{E}_{\gamma _{1},2}(t,.)u_{1}(x)\Delta \varphi \,dx \\ 
+\displaystyle\int_{\mathbb{R}^{N}}f\left( v(t,x\right) )\varphi \,dx+%
\displaystyle\int_{\mathbb{R}^{N}}\displaystyle\int_{0}^{t}\left( t-s\right)
^{\gamma _{1}-1}\tilde{E}_{\gamma _{1},\gamma _{1}}\left( t-s\right) \ast
f\left( v(s,x)\right) \Delta \varphi \,dsdx \\ 
+\displaystyle\int_{\mathbb{R}^{N}}J_{0|t}^{2-\gamma _{1}}\left( \partial
_{t}\tilde{E}_{\gamma _{1},1}(t,x)\right) u_{0}(x)\varphi _{t}\,dx+%
\displaystyle\int_{\mathbb{R}^{N}}J_{0|t}^{2-\gamma _{1}}\left( \partial
_{t}\left( t\tilde{E}_{\gamma _{1},2}(t,.)\right) u_{1}(x)-u_{1}(x)\right)
\varphi _{t}\,dx \\ 
+\displaystyle\int_{\mathbb{R}^{N}}\displaystyle\int_{0}^{t}\tilde{E}%
_{\gamma _{1},1}\left( t-s\right) \ast f\left( v(s,x)\right) ds\varphi
_{t}\,dx.%
\end{array}%
\end{equation*}%
Using the fact that $u$ is a mild solution, we obtain 
\begin{equation}
\begin{array}{c}
\frac{\partial }{\partial t}I=\displaystyle\int_{\mathbb{R}^{N}}u\Delta
\varphi dx+\displaystyle\int_{\mathbb{R}^{N}}f\left( v(t,x\right) )\varphi
dx+\displaystyle\int_{\mathbb{R}^{N}}J_{0|t}^{2-\gamma _{1}}\left( \partial
_{t}\tilde{E}_{\gamma _{1},1}(t,x)\right) u_{0}(x)\varphi _{t}\,dx \\ 
+\displaystyle\int_{\mathbb{R}^{N}}J_{0|t}^{2-\gamma _{1}}\left( \partial
_{t}\left( t\tilde{E}_{\gamma _{1},2}(t,.)\right) u_{1}(x)-u_{1}(x)\right)
\varphi _{t}\,dx \\ 
+\displaystyle\int_{\mathbb{R}^{N}}\displaystyle\int_{0}^{t}\tilde{E}%
_{\gamma _{1},1}\left( t-s\right) \ast f\left( v(s,x)\right) ds\varphi
_{t}\,dx \\ 
=\displaystyle\int_{\mathbb{R}^{N}}u\Delta \varphi \,dx+\displaystyle\int_{%
\mathbb{R}^{N}}f\left( v(t,x\right) )\varphi \,dx+\displaystyle\int_{\mathbb{%
R}^{N}}J_{0|t}^{2-\gamma _{1}}\left( u_{t}-u_{1}\right) \varphi _{t}\,dx%
\text{.}%
\end{array}
\label{i1}
\end{equation}%
On the other hand, we have 
\begin{equation}
\frac{\partial }{\partial t}I=\int_{\mathbb{R}^{N}}\frac{\partial }{\partial
t}\left[ J_{0|t}^{2-\gamma _{1}}\left( u_{t}-u_{1}\right) \right] \varphi
\,dx+\int_{\mathbb{R}^{N}}J_{0|t}^{2-\gamma _{1}}\left( u_{t}-u_{1}\right)
\varphi _{t}\,dx.  \label{i2}
\end{equation}%
Integrating both sides of (\ref{i1}) and (\ref{i2}) on $[0,T]$, and then
identifying the terms, we get 
\begin{equation*}
\int_{0}^{T}\int_{\mathbb{R}^{N}}\frac{\partial }{\partial t}%
J_{0|t}^{2-\gamma _{1}}\left( u_{t}-u_{1}\right) \varphi
\,dxdt=\int_{0}^{T}\int_{\mathbb{R}^{N}}u\Delta \varphi \,dxdt+\displaystyle%
\int_{0}^{T}\int_{\mathbb{R}^{N}}f\left( v(t,x\right) )\varphi \,dxdt\text{.}
\end{equation*}%
The formula of integration by parts allows to write 
\begin{equation*}
\int_{0}^{T}\int_{\mathbb{R}^{N}}\left( u_{t}-u_{1}\right) D_{t|T}^{\gamma
_{1}}\varphi \,dxdt=\int_{0}^{T}\int_{\mathbb{R}^{N}}u\Delta \varphi
\,dxdt+\int_{0}^{T}\int_{\mathbb{R}^{N}}f\left( v(t,x\right) )\varphi \,dxdt.
\end{equation*}%
By an analogous calculation, we can show that 
\begin{equation*}
\int_{0}^{T}\int_{\mathbb{R}^{N}}\left( v_{t}-v_{1}\right) D_{t|T}^{\gamma
_{2}}\varphi \,dxdt=\int_{0}^{T}\int_{\mathbb{R}^{N}}v\Delta \varphi
\,dxdt+\int_{0}^{T}\int_{\mathbb{R}^{N}}g\left( u(t,x\right) )\varphi \,dxdt.
\end{equation*}%
This completes the proof.
\end{proof}

Our next result concerns the blow-up of solutions of (\ref{sys1}).\qquad\ 

\begin{theorem}[Blow-up of mild solution]
\label{NEG}Let $N\geq 1,$ $p>1,$ $q>1,$ $u_{0},v_{0},u_{1},v_{1}\in
L_{loc}^{p}\left( \mathbb{R}^{N}\right) ,$ $1<\gamma _{1},\gamma _{2}<2$, be
such that $\int_{\mathbb{R}^{N}}u_{1}\left( x\right) dx>0,$ $\int_{\mathbb{R}%
^{N}}v_{1}\left( x\right) dx>0.$ If 
\begin{equation*}
\text{or}%
\begin{array}{c}
\frac{N}{2}<\min \left\{ \frac{1}{\gamma _{1}}+\frac{\gamma _{1}+p\gamma _{2}%
}{\gamma _{1}\left( pq-1\right) },\frac{1}{\gamma _{1}}+\frac{1+p}{\left(
pq-1\right) },\frac{1}{\gamma _{1}}+\frac{\gamma _{2}+q\gamma _{1}}{\gamma
_{1}\left( pq-1\right) },\frac{1-\gamma _{2}}{\gamma _{1}}+\frac{q\left(
p+1\right) }{\left( pq-1\right) }\right\} , \\ 
\\ 
\frac{N}{2}<\min \left\{ \frac{1}{\gamma _{2}}+\frac{\gamma _{1}+\gamma _{2}p%
}{\gamma _{2}\left( pq-1\right) },\frac{1-\gamma _{1}}{\gamma _{2}}+\frac{%
p\left( q+1\right) }{\left( pq-1\right) },\frac{1}{\gamma _{2}}+\frac{\gamma
_{2}+\gamma _{1}q}{\gamma _{2}\left( pq-1\right) },\frac{1}{\gamma _{2}}+%
\frac{1+q}{\left( pq-1\right) }\right\} ,%
\end{array}%
\end{equation*}%
then the mild solution $(u,v)$ of (\ref{sys1})-(\ref{initdat}) blows up in a
finite time. \ \qquad\ 
\end{theorem}

\section{Global Existence and Decay Estimates}

\textbf{Proof of Theorem} \ref{GELT}.\bigskip

The proof of Theorem \ref{GELT} proceeds in three steps. Without loss of
generality, we assume that $1<\gamma _{1}\leq \gamma _{2}<2$ and $q\geq
p\geq 1$ such that $pq>1$.\newline

\noindent \textbf{First step:} \textbf{Global existence for} $\left(
u,v\right) $ \textbf{in} $L^{s_{1}}\left( \mathbb{R}^{N}\right) \times
L^{s_{2}}\left( \mathbb{R}^{N}\right) .$ \newline

Since $pq>1$, from (\ref{critdimension}) we have for $N\geq 2$ that 
\begin{equation*}
\frac{N}{2}\geq \max \left\{ \frac{1}{\gamma _{1}}+\frac{q+1}{pq-1},\frac{1}{%
\gamma _{1}}+\frac{p\gamma _{2}+\gamma _{1}}{\gamma _{1}\left( pq-1\right) }%
\right\} .
\end{equation*}%
If $\max \left\{ \frac{1}{\gamma _{1}}+\frac{q+1}{pq-1},\frac{1}{\gamma _{1}}%
+\frac{p\gamma _{2}+\gamma _{1}}{\gamma _{1}\left( pq-1\right) }\right\} =%
\frac{1}{\gamma _{1}}+\frac{q+1}{pq-1},$ then $\frac{N}{2}\geq \frac{1}{%
\gamma _{1}}+\frac{q+1}{pq-1}$, which gives 
\begin{equation*}
1-\frac{pq-1}{q(p+1)\gamma _{2}}<1-\frac{pq-1}{2q(p+1)}<\frac{pq-1+q\gamma
_{1}+\gamma _{1}}{\gamma _{1}q(p+1)}\leq \frac{N\left( pq-1\right) }{2q(p+1)}%
.
\end{equation*}%
If $\max \left\{ \frac{1}{2}+\frac{q+1}{pq-1},\frac{1}{\gamma _{1}}+\frac{%
p\gamma _{2}+\gamma _{1}}{\gamma _{1}\left( pq-1\right) }\right\} =\frac{1}{%
\gamma _{1}}+\frac{p\gamma _{2}+\gamma _{1}}{\gamma _{1}\left( pq-1\right) }$%
. That is $\frac{1}{\gamma _{1}}+\frac{q+1}{pq-1}\leq \frac{1}{\gamma _{1}}+%
\frac{p\gamma _{2}+\gamma _{1}}{\gamma _{1}\left( pq-1\right) },$ in this
case 
\begin{equation*}
\frac{N}{2}\geq \frac{1}{\gamma _{1}}+\frac{p\gamma _{2}+\gamma _{1}}{\gamma
_{1}\left( pq-1\right) }\geq \frac{1}{\gamma _{1}}+\frac{q+1}{pq-1},
\end{equation*}%
which gives again $\frac{N\left( pq-1\right) }{2q(p+1)}>1-\frac{pq-1}{%
q(p+1)\gamma _{2}}$, and since $1-\frac{pq-1}{q(p+1)\gamma _{2}}<1$, we can
choose $\delta >0$ such that%
\begin{equation}
1-\frac{pq-1}{q(p+1)\gamma _{2}}<\delta <\min \left\{ 1,\frac{N\left(
pq-1\right) }{2q(p+1)}\right\} .  \label{delta}
\end{equation}%
We set 
\begin{equation*}
r_{1}=\frac{N\gamma _{1}\left( pq-1\right) }{2\left[ \gamma _{1}\left(
1+\delta p\right) +\gamma _{2}p\left( 1-\delta \right) \right] }\text{, }%
\qquad r_{2}=\frac{N\gamma _{2}\left( pq-1\right) }{2\left[ \gamma
_{2}\left( 1+\delta q\right) +\gamma _{1}q\left( 1-\delta \right) \right] }%
\text{,}
\end{equation*}%
\begin{equation}
\frac{1}{s_{1}}=\frac{2\delta }{N}\frac{p+1}{pq-1},\text{ }\qquad \frac{1}{%
s_{2}}=\frac{2\delta }{N}\frac{q+1}{pq-1},  \label{sonerone}
\end{equation}%
\begin{equation*}
\sigma _{1}=\frac{\left( 1-\delta \right) \left( \gamma _{1}+\gamma
_{2}p\right) }{pq-1}\text{, }\qquad \sigma _{2}=\frac{\left( 1-\delta
\right) \left( \gamma _{2}+\gamma _{1}q\right) }{pq-1}.
\end{equation*}%
Clearly, we have 
\begin{equation*}
\frac{1}{r_{1}}=\frac{2}{N\gamma _{1}}\frac{\left( 1-\delta \right) \left(
\gamma _{1}+\gamma _{2}p\right) }{pq-1}+\frac{2\delta }{N}\frac{\left(
p+1\right) }{pq-1},
\end{equation*}%
\begin{equation*}
\frac{1}{r_{2}}=\frac{2}{N\gamma _{2}}\frac{\left( 1-\delta \right) \left(
\gamma _{2}+\gamma _{1}q\right) }{pq-1}+\frac{2\delta }{N}\frac{\left(
q+1\right) }{pq-1}\text{.}
\end{equation*}%
The choice of $\delta $ gives 
\begin{equation*}
\delta >1-\frac{pq-1}{\left( \gamma _{2}+\gamma _{1}q\right) p}\;\,\text{%
implies}\;\,p\sigma _{2}=\frac{\left( 1-\delta \right) \left( \gamma
_{2}+\gamma _{1}q\right) }{pq-1}p<1,
\end{equation*}%
and%
\begin{equation*}
\delta >1-\frac{pq-1}{\left( \gamma _{1}+\gamma _{2}p\right) q}\;\,\text{%
implies}\;\,q\sigma _{1}=\frac{\left( 1-\delta \right) \left( \gamma
_{1}+\gamma _{2}p\right) }{pq-1}q<1\text{.}
\end{equation*}%
It is easy to check that 
\begin{equation*}
s_{1}>q\text{, \ }\;s_{2}>p\text{, \ }\;ps_{1}>s_{2}\text{, }\;qs_{2}>s_{1}%
\text{, \ }\;s_{1}>r_{1}>1\text{, \ }\;s_{2}>r_{2}>1,
\end{equation*}%
\begin{equation*}
\frac{N}{2}\gamma _{1}\left( \frac{1}{r_{1}}-\frac{1}{s_{1}}\right)
q<1,\qquad \frac{N}{2}\gamma _{2}\left( \frac{1}{r_{2}}-\frac{1}{s_{2}}%
\right) p<1\text{,}
\end{equation*}%
and%
\begin{equation*}
\frac{N}{2}\left( \frac{p}{s_{2}}-\frac{1}{s_{1}}\right) =\delta =\frac{N}{2}%
\left( \frac{q}{s_{1}}-\frac{1}{s_{2}}\right) .
\end{equation*}%
One can easily verify that 
\begin{equation*}
\delta >\frac{pq\left( \gamma _{1}-1\right) +1+p\gamma _{2}}{\left[ \gamma
_{1}q+\gamma _{2}\right] p}\;\;\Longleftrightarrow \;\;\left( \gamma _{2}-%
\frac{N}{2}\gamma _{2}\left( \frac{q}{s_{1}}-\frac{1}{s_{2}}\right) -q\sigma
_{1}\right) p>-1.
\end{equation*}

Let $\left( u_{0},v_{0}\right) \in L^{r_{1}}\left( \mathbb{R}^{N}\right)
\times L^{r_{2}}\left( \mathbb{R}^{N}\right) $. Let $u\in C\left( \left[
0,T_{\max }\right) ;L^{s_{1}}\left( \mathbb{R}^{N}\right) \right) \newline
$and $v\in C\left( \left[ 0,T_{\max }\right) ;L^{s_{2}}\left( \mathbb{R}%
^{N}\right) \right) $. For $t\in \lbrack 0,T_{\max })$, from (\ref{sys1}),
we have%
\begin{eqnarray}
\left\Vert u(t,.)\right\Vert _{s_{1}} &\leq &\left\Vert \tilde{E}_{\gamma
_{1},1}(t)u_{0}\right\Vert _{s_{1}}+\left\Vert t\tilde{E}_{\gamma
_{1},2}(t,.)\right\Vert _{s_{1}}  \notag \\
&&+\int_{0}^{t}(t-\tau )^{\gamma _{1}-1}\left\Vert \tilde{E}_{\gamma
_{1},\gamma _{1}}(t-\tau )\left\vert v(\tau ,.)\right\vert ^{p}\right\Vert
_{s_{1}}d\tau ,  \label{lso1}
\end{eqnarray}%
\begin{eqnarray}
\left\Vert v\left( t,.\right) \right\Vert _{s_{2}} &\leq &\left\Vert \tilde{E%
}_{\gamma _{2},1}(t)v_{0}\right\Vert _{s_{2}}+\left\Vert t\tilde{E}_{\gamma
_{2},2}(t,.)\right\Vert _{s_{2}}  \notag \\
&&+\int_{0}^{t}(t-\tau )^{\gamma _{2}-1}\left\Vert \tilde{E}_{\gamma
_{2},\gamma _{2}}(t-\tau )\left\vert u(\tau ,.)\right\vert ^{q}\right\Vert
_{s_{2}}\,d\tau .  \label{lso2}
\end{eqnarray}%
Applying Lemmas \ref{galpha} and \ref{Linfty}, we get

\begin{eqnarray}
\left\Vert u\left( t,.\right) \right\Vert _{s_{1}} &\leq &t^{-\sigma
_{1}}\left\Vert u_{0}\right\Vert _{r_{1}}+t^{-\sigma _{1}}\Vert u_{1}\Vert _{%
\mathcal{\dot{H}}_{r_{1}}^{-\frac{2}{\gamma _{1}}}}  \notag \\
&&+\,C\int_{0}^{t}\left( t-\tau \right) ^{\gamma _{1}-1}\left( t-\tau
\right) ^{-\frac{N}{2}\gamma _{1}\left( \frac{p}{s_{2}}-\frac{1}{s_{1}}%
\right) }\left\Vert v\left( \tau ,.\right) \right\Vert _{s_{2}}^{p}d\tau 
\text{,}  \label{lso3}
\end{eqnarray}%
\begin{eqnarray}
\left\Vert v\left( t,.\right) \right\Vert _{s_{2}} &\leq &t^{-\sigma
_{2}}\left\Vert v_{0}\right\Vert _{r_{2}}+t^{-\sigma _{1}}\Vert v_{1}\Vert _{%
\mathcal{\dot{H}}_{r_{2}}^{-\frac{2}{\gamma _{2}}}}  \notag \\
&&+\,C\int_{0}^{t}\left( t-\tau \right) ^{\gamma _{2}-1}\left( t-\tau
\right) ^{-\frac{N}{2}\gamma _{2}\left( \frac{q}{s_{1}}-\frac{1}{s_{2}}%
\right) }\left\Vert u\left( \tau ,.\right) \right\Vert _{s_{1}}^{q}d\tau 
\text{.}  \label{lso4}
\end{eqnarray}%
Using (\ref{lso4}) into (\ref{lso3}), we obtain%
\begin{eqnarray*}
&&\left\Vert u\left( t,.\right) \right\Vert _{s_{1}}\leq \left( \left\Vert
u_{0}\right\Vert _{r_{1}}+\Vert u_{1}\Vert _{\mathcal{\dot{H}}_{r_{1}}^{-%
\frac{2}{\gamma _{1}}}}\right) t^{-\sigma _{1}}+\,C\int_{0}^{t}\left( t-\tau
\right) ^{\gamma _{1}-1}\left( t-\tau \right) ^{-\frac{N}{2}\gamma
_{1}\left( \frac{p}{s_{2}}-\frac{1}{s_{1}}\right) }d\tau \\
&&\qquad \times \left( \left( \left\Vert v_{0}\right\Vert _{r_{2}}+\Vert
v_{1}\Vert _{\mathcal{\dot{H}}_{r_{2}}^{-\frac{2}{\gamma _{2}}}}\right)
t^{-\sigma _{2}}+\,C\int_{0}^{t}\left( t-\tau \right) ^{\gamma _{2}-1}\left(
t-\tau \right) ^{-\frac{N}{2}\gamma _{2}\left( \frac{q}{s_{1}}-\frac{1}{s_{2}%
}\right) }\left\Vert u\left( t,.\right) \right\Vert _{s_{1}}^{q}d\tau
\right) ^{p},
\end{eqnarray*}%
provided that $1-\frac{1}{\gamma _{1}}<\frac{N}{2}\left( \frac{q}{s_{1}}-%
\frac{1}{s_{2}}\right) <1$ and $1-\frac{1}{\gamma _{2}}<\frac{N}{2}\left( 
\frac{p}{s_{2}}-\frac{1}{s_{1}}\right) <1$ which are indeed satisfied.

Hence 
\begin{eqnarray}
\left\Vert u(t,.)\right\Vert _{s_{1}} &\leq &\left( \left\Vert
u_{0}\right\Vert _{r_{1}}+\Vert u_{1}\Vert _{\mathcal{\dot{H}}_{r_{1}}^{-%
\frac{2}{\gamma _{1}}}}\right) t^{-\sigma _{1}}  \notag \\
&&+\,C\displaystyle\int_{0}^{t}\left( t-\tau \right) ^{\gamma _{1}-1-\frac{N%
}{2}\gamma _{1}\left( \frac{p}{s_{2}}-\frac{1}{s_{1}}\right) }\tau
^{-p\sigma _{2}}d\tau \left( \left\Vert v_{0}\right\Vert _{r_{2}}+\Vert
v_{1}\Vert _{\mathcal{\dot{H}}_{r_{2}}^{-\frac{2}{\gamma _{2}}}}\right) ^{p}
\label{lsom} \\
&&+\,C\int_{0}^{t}\left( t-\tau \right) ^{\gamma _{1}-1-\frac{N}{2}\gamma
_{1}\left( \frac{p}{s_{2}}-\frac{1}{s_{1}}\right) }\tau ^{\left( \gamma _{2}-%
\frac{N}{2}\gamma _{2}\left( \frac{q}{s_{1}}-\frac{1}{s_{2}}\right) -q\sigma
_{1}\right) p}\left( \tau ^{^{\sigma _{1}}}\left\Vert u\left( \tau ,.\right)
\right\Vert _{s_{1}}\right) ^{pq}\,d\tau .  \notag
\end{eqnarray}%
Multiplying both sides of (\ref{lsom}) by $t^{\sigma _{1}}$ with $\sigma
_{1}=\frac{\left( 1-\delta \right) \left( \gamma _{1}+\gamma _{2}p\right) }{%
pq-1},$ we get

\begin{eqnarray}
t^{\sigma _{1}}\left\Vert u\left( t,.\right) \right\Vert _{s_{1}} &\leq
&\left\Vert u_{0}\right\Vert _{r_{1}}+\Vert u_{1}\Vert _{\mathcal{\dot{H}}%
_{r_{1}}^{-\frac{2}{\gamma _{1}}}}  \notag \\
&&+\,Ct^{\sigma _{1}}\int_{0}^{t}\left( t-\tau \right) ^{\gamma _{1}-1-\frac{%
N}{2}\gamma _{1}\left( \frac{p}{s_{2}}-\frac{1}{s_{1}}\right) }\tau
^{-p\sigma _{2}}d\tau \left( \left\Vert v_{0}\right\Vert _{r_{2}}+\Vert
v_{1}\Vert _{\mathcal{\dot{H}}_{r_{2}}^{-\frac{2}{\gamma _{2}}}}\right) ^{p}
\label{lso5} \\
&&+\,Ct^{\sigma _{1}}\int_{0}^{t}\left( t-\tau \right) ^{\gamma _{1}-1-\frac{%
N}{2}\gamma _{1}\left( \frac{p}{s_{2}}-\frac{1}{s_{1}}\right) }\tau ^{\left(
\gamma _{2}-\frac{N}{2}\gamma _{2}\left( \frac{q}{s_{1}}-\frac{1}{s_{2}}%
\right) -q\sigma _{1}\right) p}\left( \tau ^{^{\sigma _{1}}}\left\Vert
u\left( \tau ,.\right) \right\Vert _{s_{1}}\right) ^{pq}d\tau .  \notag
\end{eqnarray}

Since $\gamma _{1}-1-\frac{N}{2}\gamma _{1}\left( \frac{p}{s_{2}}-\frac{1}{%
s_{1}}\right) >-1,$ $\left( \gamma _{2}-\frac{N}{2}\gamma _{2}\left( \frac{q%
}{s_{1}}-\frac{1}{s_{2}}\right) -q\sigma _{1}\right) p>-1,$ we have 
\begin{eqnarray}
t^{\sigma _{1}}\left\Vert u\left( t,.\right) \right\Vert _{s_{1}} &\leq
&\left\Vert u_{0}\right\Vert _{r_{1}}+\Vert u_{1}\Vert _{\mathcal{\dot{H}}%
_{r_{1}}^{-\frac{2}{\gamma _{1}}}}+Ct^{\sigma _{1}+\gamma _{1}-\frac{N}{2}%
\gamma _{1}\left( \frac{p}{s_{2}}-\frac{1}{s_{1}}\right) -p\sigma
_{2}}\left( \left\Vert v_{0}\right\Vert _{r_{2}}^{p}+\Vert v_{1}\Vert _{%
\mathcal{\dot{H}}_{r_{2}}^{-\frac{2}{\gamma _{2}}}}^{p}\right)  \notag \\
&&+Ct^{\sigma _{1}+\gamma _{1}-\frac{N}{2}\gamma _{1}\left( \frac{p}{s_{2}}-%
\frac{1}{s_{1}}\right) +\left( \gamma _{2}-\frac{N}{2}\gamma _{2}\left( 
\frac{q}{s_{1}}-\frac{1}{s_{2}}\right) -q\sigma _{1}\right) p}\left(
\sup_{0\leq \tau \leq t}\tau ^{^{\sigma _{1}}}\left\Vert u\left( \tau
,.\right) \right\Vert _{s_{1}}\right) ^{pq}.  \notag \\
&&  \label{lso6}
\end{eqnarray}%
Note that%
\begin{equation*}
\sigma _{1}=\frac{N}{2}\gamma _{1}\left( \frac{1}{r_{1}}-\frac{1}{s_{1}}%
\right) \text{,}
\end{equation*}%
\begin{equation*}
\sigma _{1}+\gamma _{1}-\frac{N}{2}\gamma _{1}\left( \frac{p}{s_{2}}-\frac{1%
}{s_{1}}\right) -p\sigma _{2}=0\text{,}
\end{equation*}%
\begin{equation*}
\sigma _{1}+\gamma _{1}-\frac{N}{2}\gamma _{1}\left( \frac{p}{s_{2}}-\frac{1%
}{s_{1}}\right) +\left( \gamma _{2}-\frac{N}{2}\gamma _{2}\left( \frac{q}{%
s_{1}}-\frac{1}{s_{2}}\right) -q\sigma _{1}\right) p=0,
\end{equation*}%
\begin{equation*}
\sigma _{1}+\gamma _{1}-\gamma _{1}\delta +\left( \gamma _{2}-\gamma
_{2}\delta -q\sigma _{1}\right) p=0\text{.}
\end{equation*}%
Define $f(t)=\sup\limits_{0\leq \tau \leq t}\tau ^{^{\sigma _{1}}}\left\Vert
u\left( \tau ,.\right) \right\Vert _{s_{1}},$ $t\in \left[ 0,T_{\max
}\right) $. So we deduce from (\ref{lso5}) that%
\begin{equation}
f(t)\leq C\left( \left\Vert u_{0}\right\Vert _{r_{1}}+\Vert u_{1}\Vert _{%
\mathcal{\dot{H}}_{r_{1}}^{-\frac{2}{\gamma _{1}}}}+\left\Vert
v_{0}\right\Vert _{r_{2}}^{p}+\Vert v_{1}\Vert _{\mathcal{\dot{H}}_{r_{2}}^{-%
\frac{2}{\gamma _{2}}}}^{p}+f(t)^{pq}\right) ,\text{ }  \label{lso7}
\end{equation}%
for all $t\in \left( 0,T_{\max }\right) $.\newline
Setting 
\begin{equation*}
A=\left\Vert u_{0}\right\Vert _{r_{1}}+\Vert u_{1}\Vert _{\mathcal{\dot{H}}%
_{r_{1}}^{-\frac{2}{\gamma _{1}}}}+\left\Vert v_{0}\right\Vert
_{r_{2}}^{p}+\Vert v_{1}\Vert _{\mathcal{\dot{H}}_{r_{2}}^{-\frac{2}{\gamma
_{2}}}}^{p}\text{.}
\end{equation*}%
Now if we take $A$ small enough such that $A<\left( 2C\right) ^{\frac{pq}{%
1-pq}}$, then it follows by continuity argument that (\ref{lso7}) implies 
\begin{equation}
f(t)\leq 2CA\text{, for all }t\in \left[ 0,T_{\max }\right) .
\label{epsimpli}
\end{equation}%
Indeed, if (\ref{epsimpli}) is not true. That is to say $f(t_{0})>2CA$,
holds true for some\newline
$t_{0}\in \left( 0,T_{\max }\right) $. By the intermediate value theorem
since $f$ is continuous, non-decreasing and $f\left( 0\right) =0$, there
exists $t_{1}\in \left( 0,t_{0}\right) $ such that $f(t_{1})=2CA$. From (\ref%
{lso7}), we get 
\begin{equation}
2CA=f(t_{1})\leq C\left( A+f(t_{1})^{pq}\right) ,\text{ }  \label{nineq}
\end{equation}%
from which, it yields%
\begin{equation*}
2CA\leq C\left( A+\left( 2CA\right) ^{pq}\right) \text{,}
\end{equation*}%
which is equivalent to%
\begin{equation*}
A\geq \left( 2C\right) ^{\frac{pq}{1-pq}}\text{.}
\end{equation*}%
This is a contradiction. Therefore, it follows that%
\begin{equation}
f(t)\leq 2CA\text{, for any }t\in \left[ 0,T_{\max }\right) \text{.}
\label{inv}
\end{equation}%
Thus 
\begin{equation}
t^{\sigma _{1}}\left\Vert u\left( t,.\right) \right\Vert _{s_{1}}\leq C\text{%
, for any }t\in \left[ 0,T_{\max }\right) \text{. }  \label{lso8}
\end{equation}%
Similarly, we obtain%
\begin{equation}
t^{\sigma _{2}}\left\Vert v\left( t,.\right) \right\Vert _{s_{2}}\leq C\text{%
, for any }t\in \left[ 0,T_{\max }\right) \text{.}  \label{lso9}
\end{equation}%
Now, from (\ref{lso1}), (\ref{lso2}) and Lemma \ref{Linfty}, we can easily
see that%
\begin{equation}
\left\Vert u\left( t,.\right) \right\Vert _{\infty },\text{ }\left\Vert
v\left( t,.\right) \right\Vert _{\infty }\leq C\text{, for any }t\in \left[
0,1\right] \text{.}  \label{lso10}
\end{equation}%
On the other hand, since $s_{1}$ and $s_{2}$ satisfy 
\begin{equation*}
\frac{\left( 1-\delta \right) \left( p+1\right) s_{1}}{\left( pq-1\right)
s_{2}}\gamma _{2}<1,\text{ \ }\quad \frac{\left( 1-\delta \right) \left(
q+1\right) s_{2}}{\left( pq-1\right) s_{1}}\gamma _{2}<1,
\end{equation*}%
it follows from (\ref{lso1}), (\ref{lso2}), Lemma \ref{galpha} and Lemma \ref%
{Linfty} that 
\begin{eqnarray}
&&\left\Vert u\left( t,.\right) \right\Vert _{s_{1}}\leq \left\Vert \tilde{E}%
_{\gamma _{1},1}\left( t\right) u_{0}\right\Vert _{s_{1}}+t\left\Vert \tilde{%
E}_{\gamma _{2},2}\left( t\right) u_{1}\right\Vert _{s_{1}}\qquad  \notag \\
&&\qquad \qquad \;\;+\int_{0}^{t}\left( t-\tau \right) ^{\gamma
_{1}-1}\left\Vert \tilde{E}_{\gamma _{1},\gamma _{1}}\left( t-\tau \right)
\left\vert v\left( \tau ,.\right) \right\vert ^{p}\right\Vert _{s_{1}}d\tau
\qquad  \notag \\
&&\qquad \qquad \;\;\leq C\left\Vert u_{0}\right\Vert _{s_{1}}+t\left\Vert
u_{1}\right\Vert _{s_{1}}+C\int_{0}^{t}\left( t-\tau \right) ^{\gamma
_{1}-1}\left\Vert \left\vert v\left( \tau ,.\right) \right\vert
^{p}\right\Vert _{s_{1}}d\tau \qquad  \notag \\
&&\qquad \qquad \;\;\leq C\left\Vert u_{0}\right\Vert _{s_{1}}+\left\Vert
u_{1}\right\Vert _{s_{1}}+C\sup_{\tau \in \left( 0,t\right) }\left\Vert
v\left( \tau \right) \right\Vert _{\infty }^{p-\frac{s_{2}}{s_{1}}}%
\displaystyle\int_{0}^{t}\left( t-\tau \right) ^{\gamma _{1}-1}\left\Vert
v\left( \tau ,.\right) \right\Vert _{s_{2}}^{\frac{s_{2}}{s_{1}}}d\tau \qquad
\notag \\
&&\qquad \qquad \;\;\leq C\left\Vert u_{0}\right\Vert _{s_{1}}+\left\Vert
u_{1}\right\Vert _{s_{1}}+C\sup_{\tau \in \left( 0,t\right) }\left\Vert
v\left( \tau \right) \right\Vert _{\infty }^{p-\frac{s_{2}}{s_{1}}}%
\displaystyle\int_{0}^{t}\left\Vert v\left( \tau ,.\right) \right\Vert
_{s_{2}}^{\frac{s_{2}}{s_{1}}}d\tau  \label{lso11}
\end{eqnarray}%
for all $t\in \left[ 0,1\right] $. Hence $\left\Vert u\left( t,.\right)
\right\Vert _{s_{1}}\leq C,$ for any $t\in \left[ 0,1\right] $.\newline

Analogously, 
\begin{equation}
\left\Vert v\left( t,.\right) \right\Vert _{s_{2}}\leq C, \; \, \text{for all%
} \; \, t\in \left[ 0,1\right] .  \label{lso12}
\end{equation}

From (\ref{lso8}), (\ref{lso9}), (\ref{lso11}), (\ref{lso12}) and Lemma \ref%
{poldec}, we conclude that 
\begin{equation}
\left\{ 
\begin{array}{l}
\left\Vert u\left( t,.\right) \right\Vert _{s_{1}}\leq C\left( t+1\right) ^{-%
\frac{\left( 1-\delta \right) \left( \gamma _{1}+p\gamma _{2}\right) }{pq-1}%
},\bigskip \\ 
\left\Vert u\left( t,.\right) \right\Vert _{s_{2}}\leq C\left( t+1\right) ^{-%
\frac{\left( 1-\delta \right) \left( \gamma _{2}+q\gamma _{1}\right) }{pq-1}%
},%
\end{array}%
\right.  \label{lso13}
\end{equation}%
for all $t\in \left[ 0,T_{\max }\right) $ .

\noindent \textbf{Second step. }$L^{\infty }$\textbf{-global existence
estimates of\ }$\left( u,v\right) $ in $L^{\infty }\left( \mathbb{R}%
^{N}\right) \times L^{\infty }\left( \mathbb{R}^{N}\right) $.\newline

Let $s_{1},$ $s_{2}$ be as in (\ref{sonerone}). Since $p\leq q,$ we have 
\begin{equation*}
\frac{Np}{2s_{2}}\leq \frac{Nq}{2s_{1}}\text{.}
\end{equation*}%
We further assume, for some $\xi >q$ and $w>p,$ that $u(t)\in L^{w}\left( 
\mathbb{R}^{N}\right) ,$ $v(t)\in L^{\xi }\left( \mathbb{R}^{N}\right) ,$
and 
\begin{equation}
\left\{ 
\begin{array}{l}
\left\Vert u\left( t,.\right) \right\Vert _{w}\leq C\left(
1+t^{k_{1}}\right) ,\;\,t\in \left[ 0,T_{\max }\right) ,\bigskip \\ 
\left\Vert v\left( t,.\right) \right\Vert _{\xi }\leq C\left(
1+t^{k_{2}}\right) ,\;\,t\in \left[ 0,T_{\max }\right) ,%
\end{array}%
\right.  \label{lso14}
\end{equation}%
holds true for some positive constants $k_{1}$ and $k_{2}$. Then, by (\ref%
{lso1}), (\ref{lso2}) and Lemma \ref{Linfty}, we have 
\begin{equation*}
\qquad \left\Vert u\left( t,.\right) \right\Vert _{\infty }\leq \left\Vert 
\tilde{E}_{\gamma _{1},1}\left( t\right) u_{0}\right\Vert _{\infty
}+t\left\Vert \tilde{E}_{\gamma _{2},2}\left( t\right) u_{1}\right\Vert
_{\infty }
\end{equation*}%
\begin{equation}
\quad \quad \quad \quad \quad \quad \quad \quad +\int_{0}^{t}\left( t-\tau
\right) ^{\gamma _{1}-1-\frac{N\gamma _{1}p}{2\xi }}\left\Vert v\left( \tau
,.\right) \right\Vert _{\xi }^{p}d\tau ,  \label{lso15}
\end{equation}%
\begin{equation*}
\qquad \left\Vert v\left( t,.\right) \right\Vert _{\infty }\leq \left\Vert 
\tilde{E}_{\gamma _{2},1}\left( t\right) v_{0}\right\Vert _{\infty
}+t\left\Vert \tilde{E}_{\gamma _{2},2}\left( t\right) u_{1}\right\Vert
_{\infty }
\end{equation*}%
\begin{equation}
\quad \quad \quad \quad \quad \quad \quad \quad +\int_{0}^{t}\left( t-\tau
\right) ^{\gamma _{2}-1-\frac{N\gamma _{2}q}{2w}}\left\Vert u\left( \tau
,.\right) \right\Vert _{w}^{q}d\tau ,  \label{lso16}
\end{equation}%
for all $t\in \left[ 0,T_{\max }\right) $. If one can find $\xi $ and $w$
such that 
\begin{equation}
\frac{Np}{2\xi }<1\qquad \text{or }\qquad \frac{Nq}{2w}<1,  \label{test}
\end{equation}%
then the $L^{\infty }$-estimates of\textbf{\ }$\left( u,v\right) $ is
obtained. In fact, if $\frac{Np}{2\xi }<1,$ in view of (\ref{lso14}), it
yields from (\ref{lso15}) that%
\begin{eqnarray}
\left\Vert u\left( t,.\right) \right\Vert _{\infty } &\leq &\left\Vert 
\tilde{E}_{\gamma _{1},1}\left( t\right) u_{0}\right\Vert _{\infty
}+C\max_{\tau \in \left[ 0,t\right] }\left\Vert v\left( \tau ,.\right)
\right\Vert _{\xi }^{p}t^{\left( 1-\frac{Np}{2\xi }\right) \gamma _{1}} 
\notag \\
&\leq &C\left( 1+t^{\left( 1-\frac{Np}{2\xi }\right) \gamma
_{1}+pk_{2}}\right) ,  \label{lso17}
\end{eqnarray}%
and by taking $w=\infty $ in (\ref{lso16}), we get 
\begin{eqnarray}
\left\Vert v\left( t,.\right) \right\Vert _{\infty } &\leq &\left\Vert 
\tilde{E}_{\gamma _{2},1}\left( t\right) v_{0}\right\Vert _{\infty
}+t\left\Vert \tilde{E}_{\gamma _{2},2}\left( t\right) u_{1}\right\Vert
_{\infty }+\int_{0}^{t}\left( t-\tau \right) ^{\gamma _{2}-1}\left\Vert
u\left( \tau ,.\right) \right\Vert _{\infty }^{q}d\tau \qquad  \notag \\
\qquad \qquad &\leq &\left\Vert \tilde{E}_{\gamma _{2},1}\left( t\right)
v_{0}\right\Vert _{\infty }+t\left\Vert \tilde{E}_{\gamma _{2},2}\left(
t\right) u_{1}\right\Vert _{\infty }\qquad \qquad  \notag \\
&+&\int_{0}^{t}\left( t-\tau \right) ^{\gamma _{2}-1}\left( 1+t^{\left( 1-%
\frac{Np}{2\xi }\right) \gamma _{1}+pk_{2}}\right) ^{q}d\tau \qquad \qquad 
\notag \\
&\leq &C\left( 1+t^{\gamma _{2}+\left[ \left( 1-\frac{Np}{2\xi }\right)
\gamma _{1}+pk_{2}\right] q}\right) \text{.}  \label{lso18}
\end{eqnarray}%
These estimates show that $T_{\max }=\infty $, and 
\begin{equation}
u,v\in L_{loc}^{\infty }\left( \left[ 0,\infty \right) ;L^{\infty }\left( 
\mathbb{R}^{N}\right) \right) .  \label{test1}
\end{equation}%
In a similar manner, we can establish the case $\frac{Nq}{2w}<1$. To find
appropriate $\xi $ and $w,$ we note that (\ref{lso17}) and (\ref{lso18})
hold by taking $\xi =s_{1}$ or $w=s_{2}$ if $\frac{Nq}{2s_{1}}<1$ or $\frac{%
Np}{2s_{2}}<1$; this is certainly the case when $N\leq 2$ with $s_{1}>q$ and 
$s_{2}>p$.

Thus it remains to deal with the case $N>2$, $\frac{Nq}{2s_{1}}\geq 1$ and $%
\frac{Np}{2s_{2}}\geq 1$. We do this via an iterative process. Define $%
s_{1}^{\prime }=s_{1},$ $s_{1}^{\prime \prime }=s_{2},$ since $s_{1}^{\prime
}>q$ and $s_{1}^{\prime \prime }>p,$ using the H\"{o}lder inequality and
Lemmas \ref{galpha} and \ref{Linfty}, we get from (\ref{lso1}), (\ref{lso2})
that 
\begin{eqnarray*}
&&\qquad \left\Vert u\left( t,.\right) \right\Vert _{s_{2}^{\prime }}\leq
\left\Vert \tilde{E}_{\gamma _{1},1}\left( t\right) u_{0}\right\Vert
_{s_{2}^{\prime }}+t\left\Vert \tilde{E}_{\gamma _{1},2}\left( t\right)
u_{1}\right\Vert _{s_{2}^{\prime }}\qquad \\
&&\quad \quad \quad \quad \quad \quad \;\;+\int_{0}^{t}\left( t-\tau \right)
^{\gamma _{1}-1-\frac{N\gamma _{1}}{2}\left( \frac{p}{s_{2}^{\prime \prime }}%
-\frac{1}{s_{2}^{\prime }}\right) }\left\Vert v\left( \tau ,.\right)
\right\Vert _{s_{2}^{\prime \prime }}^{p}d\tau ,\qquad
\end{eqnarray*}%
\begin{eqnarray*}
&&\qquad \left\Vert v\left( t,.\right) \right\Vert _{s_{2}^{\prime \prime
}}\leq \left\Vert \tilde{E}_{\gamma _{2},1}\left( t\right) v_{0}\right\Vert
_{s_{2}^{\prime \prime }}+t\left\Vert \tilde{E}_{\gamma _{2},2}\left(
t\right) u_{1}\right\Vert _{s_{2}^{\prime \prime }}\qquad \\
&&\quad \quad \quad \quad \quad \quad \;\;+\int_{0}^{t}\left( t-\tau \right)
^{\gamma _{2}-1-\frac{N\gamma _{2}}{2}\left( \frac{q}{s_{1}^{\prime }}-\frac{%
1}{s_{2}^{\prime \prime }}\right) }\left\Vert u\left( \tau ,.\right)
\right\Vert _{s_{1}^{\prime }}^{q}d\tau ,\qquad
\end{eqnarray*}%
where $s_{2}^{\prime }$ and $s_{2}^{\prime \prime }$ are such that 
\begin{equation*}
\frac{N}{2}\left( \frac{p}{s_{1}^{\prime \prime }}-\frac{1}{s_{2}^{\prime }}%
\right) <1,\text{ }\qquad \frac{N}{2}\left( \frac{q}{s_{1}^{\prime }}-\frac{1%
}{s_{2}^{\prime \prime }}\right) <1.
\end{equation*}%
This can be shown by taking 
\begin{equation*}
\frac{1}{s_{2}^{\prime }}=\frac{p}{s_{1}^{\prime \prime }}-\frac{2}{N}+\eta
,\qquad \frac{1}{s_{2}^{\prime \prime }}=\frac{q}{s_{1}^{\prime }}-\frac{2}{N%
}+\eta ,
\end{equation*}%
where $0<\eta <\frac{2\left( 1-\delta \right) }{N}$ with $\delta >1-\frac{1}{%
\gamma _{1}}$. Namely%
\begin{equation*}
\frac{N}{2}\left( \frac{p}{s_{1}^{\prime \prime }}-\frac{1}{s_{2}^{\prime }}%
\right) =\frac{N}{2}\left( \frac{q}{s_{1}^{\prime }}-\frac{1}{s_{2}^{\prime
\prime }}\right) =1-\frac{N}{2}\eta >1-\frac{1}{\gamma _{1}}\text{.}
\end{equation*}%
Observe that, since $\delta >1-\frac{pq-1}{q(p+1)\gamma _{2}}>1-\frac{1}{%
\gamma _{2}}$, we have 
\begin{equation*}
1-\frac{1}{\gamma _{1}}<\frac{N}{2}\left( \frac{p}{s_{1}^{\prime \prime }}-%
\frac{1}{s_{2}^{\prime }}\right) <1,\qquad 1-\frac{1}{\gamma _{2}}<\frac{N}{2%
}\left( \frac{q}{s_{1}^{\prime }}-\frac{1}{s_{2}^{\prime \prime }}\right) <1%
\text{,}
\end{equation*}%
\begin{equation}
\frac{1}{s_{1}^{\prime }}-\frac{1}{s_{2}^{\prime }}=\frac{2}{N}\left(
1-\delta \right) -\eta >0,\qquad \frac{1}{s_{1}^{\prime \prime }}-\frac{1}{%
s_{2}^{\prime \prime }}=\frac{2}{N}\left( 1-\delta \right) -\eta >0\text{,}
\label{lso19}
\end{equation}%
and hence $s_{2}^{\prime }>s_{1}^{\prime }>q$ and $s_{2}^{\prime \prime
}>s_{1}^{\prime \prime }>p$.\newline
Next, define the sequences $\left\{ s_{i}^{\prime }\right\} _{i\geq 1}$ and $%
\left\{ s_{i}^{\prime \prime }\right\} _{i\geq 1}$, iteratively, as follows 
\begin{equation}
\frac{1}{s_{i}^{\prime }}=\frac{p}{s_{i-1}^{\prime \prime }}-\frac{2}{N}%
+\eta ,\text{ }\qquad \frac{1}{s_{i}^{\prime \prime }}=\frac{q}{%
s_{i-1}^{\prime }}-\frac{2}{N}+\eta ,\text{ }i\geq 3.  \label{lso20}
\end{equation}%
Then 
\begin{equation*}
\frac{1}{s_{i}^{\prime }}-\frac{1}{s_{i+1}^{\prime }}=p\left( \frac{1}{%
s_{i-1}^{\prime \prime }}-\frac{1}{s_{i}^{\prime \prime }}\right) =pq\left( 
\frac{1}{s_{i-2}^{\prime }}-\frac{1}{s_{i-1}^{\prime }}\right) ,
\end{equation*}%
\begin{equation*}
\frac{1}{s_{i}^{\prime \prime }}-\frac{1}{s_{i+1}^{\prime \prime }}=q\left( 
\frac{1}{s_{i-1}^{\prime }}-\frac{1}{s_{i}^{\prime }}\right) =pq\left( \frac{%
1}{s_{i-2}^{\prime \prime }}-\frac{1}{s_{i-1}^{\prime \prime }}\right) .
\end{equation*}%
Since $pq>1$, in view of (\ref{lso19}), we get 
\begin{equation}
\frac{1}{s_{i}^{\prime }}>\frac{1}{s_{i+1}^{\prime }},\text{ }\qquad \frac{1%
}{s_{i}^{\prime \prime }}>\frac{1}{s_{i+1}^{\prime \prime }}\text{, }i\geq 1%
\text{,}  \label{lso21}
\end{equation}%
and 
\begin{equation}
\lim_{i\rightarrow +\infty }\left( \frac{1}{s_{i}^{\prime }}-\frac{1}{%
s_{i+1}^{\prime }}\right) =\lim_{i\rightarrow +\infty }\left( \frac{1}{%
s_{i}^{\prime \prime }}-\frac{1}{s_{i+1}^{\prime \prime }}\right) =+\infty 
\text{.}  \label{lso22}
\end{equation}%
Now, we ensure that there exists $i_{0}$ such that 
\begin{equation}
\frac{p}{s_{i_{0}}^{\prime \prime }}<\frac{2}{N}\text{ }\qquad \text{or }%
\qquad \frac{q}{s_{i_{0}}^{\prime }}<\frac{2}{N}.  \label{lso23}
\end{equation}%
On the contrary, that is, $\frac{p}{s_{i}^{\prime \prime }}\geq \frac{2}{N}$
and $\frac{q}{s_{i}^{\prime }}\geq \frac{2}{N}$ for all $i\geq 1.$ Then, by (%
\ref{lso20}), we see that $s_{i}^{\prime }>0,$ $s_{i}^{\prime \prime }>0$
for all $i\geq 1$ and hence, by (\ref{lso21}), 
\begin{equation*}
q<s_{1}^{\prime }<...<s_{i}^{\prime }<...,\text{ }p<s_{1}^{\prime \prime
}<...<s_{i}^{\prime \prime }<....
\end{equation*}%
which contradicts \eqref{lso22}.

Let $i_{0}$ be the smallest number satisfying (\ref{lso23}). Notice that $%
i_{0}\geq 2$. Without loss of generality, we assume that 
\begin{equation}
\frac{p}{s_{i_{0}}^{\prime \prime }}<\frac{2}{N},\text{ \ }\frac{p}{%
s_{i}^{\prime \prime }}\geq \frac{2}{N}\text{ for any }1\leq i\leq
i_{0}-1,\qquad \frac{q}{s_{i}^{\prime }}\geq \frac{2}{N}\text{ for any }%
1\leq i\leq i_{0}\text{.}  \label{lso24}
\end{equation}%
It then follows from \eqref{lso20} that 
\begin{equation*}
s_{i}^{\prime }>0\text{ for any }1\leq i\leq i_{0},\qquad s_{i}^{\prime
\prime }>0\text{ for any }1\leq i\leq i_{0}+1\text{,}
\end{equation*}%
which together with (\ref{lso21}) leads to 
\begin{equation*}
q<...<s_{i_{0}-1}^{\prime }<s_{i_{0}}^{\prime },\qquad
p<...<s_{i_{0}}^{\prime \prime }<s_{i_{0}+1}^{\prime \prime }\text{.}
\end{equation*}%
Now, from \eqref{lso20}, we have, for all $i\geq 2$, 
\begin{equation*}
\frac{N}{2}\left( \frac{p}{s_{i-1}^{\prime \prime }}-\frac{1}{s_{i}^{\prime }%
}\right) =1-\frac{N}{2}\eta =\frac{N}{2}\left( \frac{q}{s_{i-1}^{\prime }}-%
\frac{1}{s_{i}^{\prime \prime }}\right) \text{.}
\end{equation*}%
Now, let us deal with the boundedness of $\left( u(t,.),v(t,.)\right) $ in $%
L^{s_{i}^{\prime }}\left( \mathbb{R}^{N}\right) \times L^{s_{i}^{\prime
\prime }}\left( \mathbb{R}^{N}\right) $. Using the H\"{o}lder inequality and
Lemmas \ref{galpha}, \ref{Linfty}, it follows from \eqref{ms1}-\eqref{ms2},
inductively, that 
\begin{eqnarray}
\left\Vert u\left( t,.\right) \right\Vert _{s_{i}^{\prime }} &\leq
&\left\Vert \tilde{E}_{\gamma _{1},1}\left( t\right) u_{0}\right\Vert
_{s_{i}^{\prime }}+t\left\Vert \tilde{E}_{\gamma _{2},2}\left( t\right)
u_{1}\right\Vert _{s_{i}^{\prime }}  \notag \\
&+&C\displaystyle\int_{0}^{t}\left( t-\tau \right) ^{\gamma _{1}-1-\frac{N}{2%
}\gamma _{1}\left( \frac{p}{s_{i-1}^{\prime \prime }}-\frac{1}{s_{i}^{\prime
}}\right) }\left\Vert v\left( \tau ,.\right) \right\Vert _{s_{i-1}^{\prime
\prime }}^{p}d\tau  \notag \\
&\leq &C\left\Vert u_{0}\right\Vert _{s_{i}^{\prime }}+t\left\Vert
u_{1}\right\Vert _{s_{i}^{\prime }}  \notag \\
&+&C\displaystyle\int_{0}^{t}\left( t-\tau \right) ^{\gamma _{1}-1-\gamma
_{1}\left( 1-\frac{N}{2}\eta \right) }\left\Vert v\left( \tau ,.\right)
\right\Vert _{s_{i-1}^{\prime \prime }}^{p}d\tau ,\text{ }  \label{lso25}
\end{eqnarray}%
for any $2\leq i\leq i_{0},$ $t\in \left( 0,T_{\max }\right) $ and 
\begin{eqnarray}
\left\Vert v\left( t,.\right) \right\Vert _{s_{i}^{\prime \prime }} &\leq
&\left\Vert \tilde{E}_{\gamma _{2},1}\left( t\right) v_{0}\right\Vert
_{s_{i}^{\prime \prime }}+t\left\Vert \tilde{E}_{\gamma _{2},2}\left(
t\right) v_{1}\right\Vert _{s_{i}^{\prime \prime }}  \notag \\
&+&C\int_{0}^{t}\left( t-\tau \right) ^{\gamma _{2}-1+\frac{N}{2}\gamma
_{2}\left( \frac{q}{s_{i-1}^{\prime }}-\frac{1}{s_{i}^{\prime \prime }}%
\right) }\left\Vert u\left( \tau ,.\right) \right\Vert _{s_{i-1}^{\prime
}}^{q}d\tau  \notag \\
&\leq &C\left\Vert v_{0}\right\Vert _{s_{i}^{\prime \prime }}+C\left\Vert
v_{1}\right\Vert _{s_{i}^{\prime \prime }}  \notag \\
&+&C\int_{0}^{t}\left( t-\tau \right) ^{\gamma _{2}-1-\gamma _{2}\left( 1-%
\frac{N\eta }{2}\right) }\left\Vert u\left( \tau ,.\right) \right\Vert
_{s_{i-1}^{\prime }}^{q}d\tau ,\text{ }  \label{lso26}
\end{eqnarray}%
for any $t\in \left( 0,T_{\max }\right) $ and for any $2\leq i\leq i_{0}+1$.%
\newline
It clearly follows from \eqref{lso25} and \eqref{lso26} that $u\left(
t\right) \in L^{s_{i}^{\prime }}\left( \mathbb{R}^{N}\right) ,$ $v\left(
t\right) \in L^{s_{i}^{\prime \prime }}\left( \mathbb{R}^{N}\right) $ :%
\newline
\begin{equation}
\left\{ 
\begin{array}{l}
u\left( t,.\right) \in L^{s_{i}^{\prime }}\left( \mathbb{R}^{N}\right) ,%
\text{ }\left\Vert u\left( t,.\right) \right\Vert _{s_{i}^{\prime }}\leq
C\left( 1+t^{a_{i}}\right) ,1\leq \forall i\leq i_{0},\text{ }t\in \left(
0,T_{\max }\right) , \\[5pt] 
v\left( t,.\right) \in L^{s_{i}^{\prime \prime }}\left( \mathbb{R}%
^{N}\right) ,\text{ }\left\Vert v\left( t,.\right) \right\Vert
_{s_{i}^{\prime \prime }}\leq C\left( 1+t^{b_{i}}\right) ,1\leq \forall
i\leq i_{0}+1,\text{ }t\in \left( 0,T_{\max }\right) ,%
\end{array}%
\right.  \label{estlprim}
\end{equation}%
for some positive constants $a_{i},$ $b_{i}.$ Since $\frac{Np}{%
2s_{i_{0}^{\prime \prime }}}<1,$ taking $s_{2}=s_{i_{0}}^{\prime \prime },$ (%
\ref{test}) holds. In consequence, we get $T_{\max }=+\infty $ and that (\ref%
{test1}) holds.

\noindent \textbf{3.} $L^{\infty }$\textbf{-decay estimates.}

Let 
\begin{equation*}
\sigma _{1}=\frac{\left( 1-\delta \right) \left( p\gamma _{2}+\gamma
_{1}\right) }{\left( pq-1\right) }\text{, \ \ \ }\sigma _{2}=\frac{\left(
1-\delta \right) \left( q\gamma _{1}+\gamma _{2}\right) }{\left( pq-1\right) 
}.
\end{equation*}%
If $\frac{pN}{2s_{2}}<1$, by taking $\xi =s_{2}$ in \eqref{lso16} and using (%
\ref{lso13}), we get 
\begin{equation}
\left\Vert u\left( t,.\right) \right\Vert _{\infty }\leq Ct^{-\frac{N\gamma
_{1}}{2r_{1}}}\left\Vert u_{0}\right\Vert _{r_{1}}+Ct^{1-\frac{N\gamma _{1}}{%
2m}}\left\Vert u_{1}\right\Vert _{m}+C\displaystyle\int_{0}^{t}\left( t-\tau
\right) ^{\gamma _{1}-1-\frac{N\gamma _{1}}{2}\frac{p}{s_{2}}}\tau
^{-p\sigma _{2}}d\tau \text{.}  \label{lso27}
\end{equation}%
From (\ref{critdimension}) with $pq>q+2$, we get $2\left( 1+p\right) -\left(
pq-1\right) <\frac{N\left( pq-1\right) }{2q}$ which implies that $\frac{N}{%
2r_{1}}<1$ and for any $m$ depending on $N$ such that $\frac{N}{2}<m<\frac{%
N\gamma _{1}}{2},$ $N\geq 2$, we infer that 
\begin{equation*}
1-\frac{N\gamma _{1}}{2m}<0\text{ \ \ and \ }\frac{N}{2m}<1.
\end{equation*}%
On the other hand, since 
\begin{equation*}
p\sigma _{2}<1,\text{ }\quad \gamma _{1}-\frac{N\gamma _{1}}{2}\frac{p}{s_{2}%
}-p\sigma _{2}=-\frac{\left[ \gamma _{1}+\gamma _{1}p\delta +\left( 1-\delta
\right) p\gamma _{2}\right] }{pq-1},
\end{equation*}%
and 
\begin{equation}
\frac{\gamma _{1}+\gamma _{1}p\delta +p\gamma _{2}\left( 1-\delta \right) }{%
pq-1}=\frac{N\gamma _{1}}{2r_{1}},  \label{comp}
\end{equation}%
it follows from \eqref{lso27} and \eqref{comp} that 
\begin{equation}
\left\Vert u\left( t,.\right) \right\Vert _{\infty }\leq Ct^{-\frac{N}{2r_{1}%
}\gamma _{1}}+Ct^{1-\frac{N}{2m}\gamma _{1}}+Ct^{-\frac{\left[ \gamma
_{1}+\gamma _{1}p\delta +\left( 1-\delta \right) p\gamma _{2}\right] }{pq-1}%
}.  \label{lso28}
\end{equation}%
Therefore, we have from \eqref{estlprim}, \eqref{lso28} and Lemma \ref%
{poldec} that 
\begin{equation*}
\left\Vert u\left( t,.\right) \right\Vert _{\infty }\leq C\left( 1+t\right)
^{-\min \left\{ \frac{N}{2r_{1}}\gamma _{1},\frac{N}{2m}\gamma
_{1}-1\right\} }\text{, for any }t\geq 0.
\end{equation*}%
Similarly, for $\frac{qN}{2s_{1}}<1$ we find that 
\begin{equation}
\left\Vert v\left( t,.\right) \right\Vert _{\infty }\leq C\left( 1+t\right)
^{-\min \left\{ \frac{N}{2r_{2}}\gamma _{2},1-\frac{N}{2m}\gamma
_{2}\right\} }\text{, for any }t\geq 0\text{.}  \label{lso29}
\end{equation}%
Also, (\ref{lso28}) holds as $pN/\left( 2s_{2}\right) \leq qN/\left(
2s_{1}\right) .$

In particular, if $pq>\gamma _{2}\left( q+1\right) +1$, we can choose $%
\delta >1-\frac{pq-1}{q(p+1)\gamma _{2}}$ and $\delta \approx 1-\frac{pq-1}{%
q(p+1)\gamma _{2}}$ such that $qN/\left( 2s_{1}\right) <1$. Therefore, the
estimates (\ref{lso28}) and (\ref{lso29}) hold. It is useful to note that $%
N\leq 2$ implies $qN/\left( 2s_{1}\right) <1$ and $qN/\left( 2s_{1}\right)
<1 $ implies $pq>\gamma _{2}\left( q+1\right) +1$.

It remains to consider the following two cases:

$\triangleright$ $N>2,$ $\frac{Np}{2s_{2}}<1$ and $\frac{Nq}{2s_{1}}\geq 1.$

Let 
\begin{equation*}
\sigma ^{\prime }=\frac{\gamma _{1}+\gamma _{1}p\delta +\left( 1-\delta
\right) p\gamma _{2}}{pq-1}.
\end{equation*}%
For positive $\mu $ such that $\mu <\min \left\{ \sigma ^{\prime },\sigma
_{1}\right\} $ and $q\mu <1$; Since $N>2$ and $q>1,$ we can choose $k>0$
such that $k>\frac{qN}{2}$ and $q\mu +\frac{qN\gamma _{2}}{2k}>\gamma _{2}.$
Since $s_{1}\leq qN/2,$ we have $k>s_{1}$. By the interpolation inequality, 
\begin{equation*}
\left\Vert u\left( t\right) \right\Vert _{k}\leq \left\Vert u\left( t\right)
\right\Vert _{\infty }^{\left( k-s_{1}\right) /k}\left\Vert u\left( t\right)
\right\Vert _{s_{1}}^{s_{1}/k}\leq Ct^{-\sigma ^{\prime }\left(
k-s_{1}\right) /k}t^{-\sigma _{1}s_{1}/k},\text{ for any }t>0.
\end{equation*}%
Therefore, by \eqref{lso8}, \eqref{lso28}, we have 
\begin{equation*}
\left\Vert u\left( t\right) \right\Vert _{k}\leq Ct^{-\mu },\;\;\text{for
all }\,t>0.
\end{equation*}%
Consequently, for any $t>0$, 
\begin{equation}
\begin{array}{l}
\left\Vert v\left( t\right) \right\Vert _{\infty }\leq \left\Vert \tilde{E}%
_{\gamma _{2},1}\left( t\right) v_{0}\right\Vert _{\infty }+t\left\Vert 
\tilde{E}_{\gamma _{2},2}\left( t\right) v_{1}\right\Vert _{\infty }+C%
\displaystyle\int_{0}^{t}\left( t-\tau \right) ^{\gamma _{2}-1-\frac{Nq}{2k}%
\gamma _{2}}\left\Vert u\left( \tau \right) \right\Vert _{k}^{q}d\tau \\ 
\qquad \qquad \leq Ct^{-\frac{N}{2r_{2}}\gamma _{2}}\left\Vert
v_{0}\right\Vert _{r_{2}}+Ct^{1-\frac{N}{2r_{2}}\gamma _{2}}\left\Vert
v_{1}\right\Vert _{r_{2}}+C\displaystyle\int_{0}^{t}\left( t-\tau \right)
^{\gamma _{2}-1-\frac{N\gamma _{2}q}{2k}}\tau ^{-q\mu }d\tau , \\ 
\qquad \qquad \leq C\left( t^{-\frac{N}{2}\gamma _{2}}+t^{1-\frac{N}{2r_{2}}%
\gamma _{2}}+t^{\gamma _{2}-\frac{N\gamma _{2}q}{2k}-q\mu }\right) \\ 
\qquad \qquad \leq Ct^{-\alpha },%
\end{array}
\label{lso30}
\end{equation}%
where $\alpha =\min \left\{ \frac{N}{2r_{2}}\gamma _{2}-1,-\gamma _{2}+\frac{%
N\gamma _{2}q}{2k}+q\mu \right\} >0,$\newline

\begin{center}
$k>s_{1}, \quad q\mu <1, \quad k>q, \quad \gamma _{2}-\frac{Nq\gamma _{2}}{2k%
}>0, \;\; \gamma _{2}-\frac{Nq\gamma _{2}}{2k}-q\mu <0.$
\end{center}

\vskip.3cm From (\ref{lso10}) and (\ref{lso30}), we infer that 
\begin{equation*}
\left\Vert v\left( t\right) \right\Vert _{\infty }\leq C\left( 1+t\right)
^{-\alpha }, \; \; \text{for all}\; \; t\geq 0.
\end{equation*}
In case $p=1$ and $q^{2}>1+4q,$ we can choose \newline

$\delta >(1+3q)/(p+1)q\gamma _{2}=(1+3q)/\left( 2\gamma _{2}q\right) $ and $%
\delta \approx (1+3q)/\left( 2\gamma _{2}q\right) $ such that $N/(2s_{2})<1.$

Thus we obtain the estimate (\ref{lso28}).

$\triangleright $ \textbf{The case: }$N>2$\textbf{$,$ }$qN/(2s_{1})\geq 1$%
\textbf{$,$ }$pN/(2s_{2})\geq 1$\textbf{$,$ }$q\geq p>1$\textbf{\ }and $%
\gamma _{1}\leq \gamma _{2}$.\newline

This case needs a careful handling and we need to restrict further the
choice of $\delta $. As $\sqrt{\frac{\left( p+1\right) q\gamma _{1}}{\left(
q+1\right) p}}<\gamma _{1}\leq \gamma _{2}<2$, $pq>1$, it follows that $1-%
\frac{pq-1}{q(p+1)\gamma _{2}}<1-\frac{\left( pq-1\right) }{p\left(
q+1\right) \gamma _{1}^{2}}$. We can select $\delta $ such that 
\begin{equation*}
1-\frac{pq-1}{q(p+1)\gamma _{2}}<\delta <\min \left\{ \frac{N\left(
pq-1\right) }{2\left( p+1\right) q},1-\frac{pq-1}{p\left( q+1\right) \gamma
_{1}^{2}}\right\} .
\end{equation*}%
Then we get immediately that $p\sigma _{2}>1/\gamma _{1}>1/q\gamma _{1}$ and 
$q\sigma _{1}>1/\gamma _{1}>1/p\gamma _{2}$.\newline
Further, we notice that there exist $\varepsilon \in \left( 0,1\right) $ and 
$\beta <1$ close to $1$ such that 
\begin{equation}
p\sigma _{2}-\varepsilon >1/\gamma _{1}>1/q\gamma _{1}\text{, \ }q\sigma
_{1}-\varepsilon >1/\gamma _{2}>1/p\gamma _{2}\text{, }\;\text{ and }%
\;\;1/\gamma _{1}<\beta -\varepsilon .  \label{epsilonbeta}
\end{equation}%
Letting $\eta =2\varepsilon \left( 1-\delta \right) /N$, we find the integer 
$i_{0}$ as in the Step 2, and, without loss of generality, assume that (\ref%
{lso24}) holds. We choose $\beta $ in addition to (\ref{epsilonbeta})
satisfying 
\begin{equation*}
\gamma _{1}<\gamma _{1}\frac{pN}{2s_{i_{0}}^{\prime \prime }}+\beta \text{,}%
\;\,\text{ \ since }1-\frac{1}{\gamma _{1}}<\frac{pN}{2s_{i_{0}}^{\prime
\prime }}\text{.}
\end{equation*}%
As 
\begin{equation*}
\delta <\frac{N\left( pq-1\right) }{2\left( p+1\right) q}\leq \frac{N\left(
pq-1\right) }{2\left( q+1\right) p},\;\;\text{and}\;\;\beta <1,
\end{equation*}%
we have 
\begin{equation}
\beta +\frac{\left( p+1\right) q\delta }{\left( pq-1\right) }<1+\frac{N}{2}%
\text{,}\;\;\text{ }\beta +\frac{\left( q+1\right) p\delta }{\left(
pq-1\right) }<1+\frac{N}{2}\text{.}  \label{betadelt}
\end{equation}%
For $2\leq i\leq i_{0}-1,$ define $r_{i+1}^{\prime }$ and $r_{i+1}^{\prime
\prime }$, inductively, as follows: 
\begin{eqnarray*}
\frac{1}{r_{2}^{\prime }} &=&\frac{1}{s_{2}^{\prime }}+\frac{2}{N}\left(
p\sigma _{2}-\varepsilon \left( 1-\delta \right) \right) \text{,}\quad \quad 
\frac{1}{r_{2}^{\prime \prime }}=\frac{1}{s_{2}^{\prime \prime }}+\frac{2}{N}%
\left( q\sigma _{1}-\varepsilon \left( 1-\delta \right) \right) \text{,} \\
\frac{1}{r_{i+1}^{\prime }} &=&\frac{1}{s_{i+1}^{\prime }}+\frac{2}{N}\left(
\beta -\varepsilon \left( 1-\delta \right) \right) ,\quad \quad \frac{1}{%
r_{i+1}^{\prime \prime }}=\frac{1}{s_{i+1}^{\prime \prime }}+\frac{2}{N}%
\left( \beta -\varepsilon \left( 1-\delta \right) \right) . \\
&&\hspace{-3cm}
\end{eqnarray*}%
It is clear that $r_{i}^{\prime },$ $r_{i}^{\prime \prime }>0$ and $%
r_{i}^{\prime }<s_{i}^{\prime },$ $r_{i}^{\prime \prime }<s_{i}^{\prime
\prime }$ for all $2\leq i\leq i_{0}.$ A simple calculation shows that $%
r_{i}^{\prime },$ $r_{i}^{\prime \prime }>1$.

As $s_{i}^{\prime }$ and $s_{i}^{\prime \prime }$ are increasing in $i$ for $%
1\leq i\leq i_{0},$ we have 
\begin{eqnarray*}
\frac{1}{r_{i+1}^{\prime }} &<&\frac{1}{s_{2}^{\prime }}+\frac{2}{N}\left(
\beta -\varepsilon \left( 1-\delta \right) \right) \\
&=&\frac{p}{s_{1}^{\prime \prime }}-\frac{2}{N}+\frac{2}{N}\varepsilon
\left( 1-\delta \right) +\frac{2}{N}\left( \beta -\varepsilon \left(
1-\delta \right) \right) \qquad \qquad \\
&=&\frac{2}{N}\left( \frac{p\left( q+1\right) \delta }{pq-1}+\beta -1\right)
<1,
\end{eqnarray*}%
from (\ref{betadelt}), i.e. $r_{i+1}^{\prime }>1.$

Similarly, we can find that $r_{i+1}^{\prime \prime }>1$.

From (\ref{test1}) and (\ref{estlprim}), we infer that there exists a
positive constant $C$ such that, for any $0\leq t\leq 1$, 
\begin{equation}
\qquad \qquad \left\Vert u(t)\right\Vert _{\infty },\text{ }\left\Vert
v(t)\right\Vert _{\infty },\text{ }\left\Vert u(t)\right\Vert _{k_{1}},\text{
}\left\Vert v(t)\right\Vert _{k_{2}}\leq C,\text{ }s_{1}^{\prime }\leq
k_{1}\leq s_{i_{0}}^{\prime },\text{ }s_{1}^{\prime \prime }\leq k_{2}\leq
s_{i_{0}}^{\prime \prime }\text{. }  \notag
\end{equation}%
Further, since $1-\eta N/2=1-\varepsilon \left( 1-\delta \right) $ and $%
p\sigma _{2}<1,$ using (\ref{lso25}), (\ref{lso26}), (\ref{lso8}) and (\ref%
{lso9}), we arrive at the estimate 
\begin{eqnarray*}
\left\Vert u\left( t,.\right) \right\Vert _{s_{2}^{\prime }} &\leq
&\left\Vert \tilde{E}_{\gamma _{1},1}\left( t\right) u_{0}\right\Vert
_{s_{2}^{\prime }}+t\left\Vert \tilde{E}_{\gamma _{1},2}\left( t\right)
u_{1}\right\Vert _{s_{2}^{\prime }} \\
&+&C\displaystyle\int_{0}^{t}\left( t-\tau \right) ^{\gamma _{1}-1-\gamma
_{1}\left( 1-\varepsilon \left( 1-\delta \right) \right) }\left\Vert u\left(
\tau ,.\right) \right\Vert _{s_{1}^{\prime \prime }}^{p}d\tau \text{,}
\end{eqnarray*}%
from which, we get%
\begin{eqnarray*}
\left\Vert u\left( t,.\right) \right\Vert _{s_{2}^{\prime }} &\leq &Ct^{-%
\frac{N}{2}\gamma _{1}\left( \frac{1}{r_{2}^{\prime }}-\frac{1}{%
s_{2}^{\prime }}\right) }\left\Vert u_{0}\right\Vert _{r_{2}^{\prime }}+t^{-%
\frac{N}{2}\gamma _{1}\left( \frac{1}{r_{2}^{\prime }}-\frac{1}{%
s_{2}^{\prime }}\right) }\left\Vert u_{1}\right\Vert _{\mathcal{\dot{H}}%
_{r_{2}^{\prime }}^{-\frac{2}{\gamma _{1}}}} \\
&+&C\displaystyle\int_{0}^{t}\left( t-\tau \right) ^{\gamma _{1}-1-\gamma
_{1}\left( 1-\varepsilon \left( 1-\delta \right) \right) }\tau ^{-p\sigma
_{2}}d\tau .\qquad \qquad
\end{eqnarray*}%
Therefore%
\begin{eqnarray*}
\left\Vert u\left( t,.\right) \right\Vert _{s_{2}^{\prime }} &\leq
&Ct^{-\gamma _{1}\left( p\sigma _{2}-\varepsilon \left( 1-\delta \right)
\right) }\left\Vert u_{0}\right\Vert _{r_{2}^{\prime }}+t^{-\gamma
_{1}\left( p\sigma _{2}-\varepsilon \left( 1-\delta \right) \right)
}\left\Vert u_{1}\right\Vert _{\mathcal{\dot{H}}_{r_{2}^{\prime }}^{-\frac{2%
}{\gamma _{1}}}} \\
&+&C\displaystyle\int_{0}^{t}\left( t-\tau \right) ^{\gamma _{1}-1-\gamma
_{1}\left( 1-\varepsilon \left( 1-\delta \right) \right) }\tau ^{-p\sigma
_{2}}d\tau \\
&\leq &Ct^{-\gamma _{1}\left( p\sigma _{2}-\varepsilon \left( 1-\delta
\right) \right) },\text{ for any }t>0\text{.}
\end{eqnarray*}%
Similarly, 
\begin{equation*}
\left\Vert v\left( t,.\right) \right\Vert _{s_{2}^{\prime \prime }}\leq
Ct^{-\gamma _{2}\left( q\sigma _{1}-\varepsilon \left( 1-\delta \right)
\right) },\text{ for any }t>0.
\end{equation*}%
In view of (\ref{epsilonbeta}) and $\beta <1,$ thanks to Lemma \ref{poldec},
for any $t>0,$ we conclude that

\begin{equation}
\left\Vert u\left( t,.\right) \right\Vert _{s_{2}^{\prime }}\leq Ct^{-\gamma
_{1}\beta /q}\text{ \ \ and \ \ }\left\Vert v\left( t,.\right) \right\Vert
_{s_{2}^{\prime \prime }}\leq Ct^{-\gamma _{2}\beta /p}\text{.}
\label{betgamma}
\end{equation}%
An iterative argument leads to 
\begin{equation*}
\left\Vert u\left( t,.\right) \right\Vert _{s_{i_{0}}^{\prime }}\leq
Ct^{-\gamma _{1}\left( \beta -\varepsilon \left( 1-\delta \right) \right)
}\leq Ct^{-\beta /q}\text{, }\left\Vert v\left( t,.\right) \right\Vert
_{s_{i_{0}}^{\prime \prime }}\leq Ct^{-\gamma _{2}\left( \beta -\varepsilon
\left( 1-\delta \right) \right) }\leq Ct^{-\beta /p},
\end{equation*}%
for any $t\geq 1$. Therefore, by \eqref{lso15} and \eqref{lso16}, we have

\begin{eqnarray*}
\left\Vert u\left( t,.\right) \right\Vert _{\infty } &\leq &Ct^{-\frac{N}{%
2r_{1}}\gamma _{1}}\left\Vert u_{0}\right\Vert _{r_{1}}+Ct^{1-\frac{N}{2m}%
\gamma _{1}}\left\Vert u_{1}\right\Vert _{m}+C\displaystyle%
\int_{0}^{t}\left( t-\tau \right) ^{\gamma _{1}-1-\gamma _{1}\frac{pN}{%
2s_{i_{0}}^{\prime \prime }}}\left\Vert v\left( \tau ,.\right) \right\Vert
_{s_{i_{0}}^{\prime \prime }}^{p}d\tau \\
&\leq &Ct^{-\frac{N}{2r_{1}}\gamma _{1}}\left\Vert u_{0}\right\Vert
_{r_{1}}+Ct^{1-\frac{N}{2m}\gamma _{1}}\left\Vert u_{1}\right\Vert _{m}+C%
\displaystyle\int_{0}^{t}\left( t-\tau \right) ^{\gamma _{1}-1-\gamma _{1}%
\frac{pN}{2s_{i_{0}}^{\prime \prime }}}\tau ^{-\beta }d\tau .
\end{eqnarray*}%
So%
\begin{eqnarray*}
\qquad \left\Vert u\left( t,.\right) \right\Vert _{\infty } &\leq &C\left(
t^{-\frac{N}{2r_{1}}\gamma _{1}}+t^{1-\frac{N}{2m}\gamma _{1}}+t^{\gamma
_{1}-\gamma _{1}\frac{pN}{2s_{i_{0}}^{\prime \prime }}-\beta }\right) \qquad
\\
&\leq &Ct^{-\tilde{\sigma}}\text{, }
\end{eqnarray*}%
where $\tilde{\sigma}=\min \left\{ \frac{N}{2r_{1}}\gamma _{1},\frac{N}{2m}%
\gamma _{1}-1,\gamma _{1}\frac{pN}{2s_{i_{0}}^{\prime \prime }}-\gamma
_{1}+\beta \right\} >0$ from (\ref{betgamma}).\newline

In view of the fact that $\frac{Nq}{2s_{1}}\geq 1$, we can make use of the
arguments similar to the ones employed for the case $\frac{Np}{2s_{2}}<1$
and $\frac{Nq}{2s_{1}}\geq 1$ to obtain $\left\Vert v\left( t,.\right)
\right\Vert _{\infty }\leq Ct^{-\hat{\sigma}}$ for some $\hat{\sigma}>0$ and
for every $t>0$. This completes the proof.\newline

\begin{remark}
In the particular case: $N>2$, $qN/(2s_{1})\geq 1$, $pN/(2s_{2})\geq 1,$ $%
q>p=1$ and $q\leq 3$, using the above method, we obtain 
\begin{equation*}
\left\Vert u\left( t,.\right) \right\Vert _{\infty }\leq Ct^{-\tilde{\sigma}}%
\text{, \ for any }t>0\text{,}
\end{equation*}%
where $\tilde{\sigma}=\min \left\{ \frac{N}{2}\gamma _{1},\frac{N}{2m}\gamma
_{1}-1,\frac{pN}{2s_{i_{0}}^{\prime \prime }}\gamma _{1}-\gamma _{1}+\gamma
_{2}\left( \beta -\varepsilon \left( 1-\delta \right) \right) \right\} $.
Here, $\varepsilon >0$ can be arbitrarily small, and $\beta $ can be
arbitrarily close to $1$. However, since $s_{i_{0}}^{\prime \prime }$
depends on $\varepsilon $ and $s_{i_{0}}^{\prime \prime }$ is decreasing in $%
\varepsilon $, it is not clear that $\tilde{\sigma}$ positive.
\end{remark}

\vskip.2cm

\begin{proof}[Proof of Theorem \protect\ref{NEG}]
The proof proceeds by contradiction. Suppose that $(u,v)$ is a mild solution
of (\ref{sys1}) which exists globally in time. Set 
\begin{equation*}
\varphi \left( t, x\right) =\varphi _{1}\left( x\right) \varphi _{2}\left(
t\right) ,\text{ }
\end{equation*}%
where $\varphi _{1}\left( x\right) =\Phi ^{l}\left( \frac{\left\vert
x\right\vert }{T^{\lambda }}\right)$ with $\Phi \in C_{0}^{\infty }\left( 
\mathbb{R}\right) $, $0\leq \Phi \left( z\right) \leq 1$, that satisfies 
\begin{equation*}
\Phi \left( z\right) =\left\{ 
\begin{array}{c}
1\text{ if }\left\vert z\right\vert \leq 1, \\[4pt] 
0\text{ if }\left\vert z\right\vert >2,%
\end{array}
\right.
\end{equation*}
and 
\begin{equation*}
\varphi _{2}\left( t\right) =\left\{ 
\begin{array}{l}
\left( 1-\frac{t}{T}\right) ^{l} \ \ \ \ \ \ \text{ if } \, t\leq T, \\[4pt] 
\text{ \ \ \ \ \ \ }0\text{\ \ \ \ \ \ \ \ \ \ if } \, t>T,%
\end{array}
\right.
\end{equation*}
where $l>\max \left\{ 1,\frac{q}{q-1}\gamma _{1}-1,\frac{p}{p-1}\gamma
_{2}-1\right\} $ and $\lambda >0$ to be determined later. \newline
We set $Q_{T}:=\mathbb{R}^{N}\times \left[ 0,T \right] .$\newline

From the definition \ref{Weaks} of the weak solution (\ref{lso16}), we have 
\begin{eqnarray}
&&\displaystyle\int_{Q_{T}}uD_{t|T}^{\gamma _{1}}\varphi \left( t,x\right)
dxdt-\displaystyle\int_{Q_{T}}u\Delta \varphi \left( t,x\right) \, dxdt = %
\displaystyle\int_{\mathbb{R}^{N}}u_{0}\left( x\right) \left(
D_{t|T}^{\gamma _{1}-1}\varphi \right) \left( 0,.\right) dx  \notag \\
&&+\displaystyle\int_{Q_{T}}u_{1}\left( x\right) D_{t|T}^{\gamma
_{1}-1}\varphi \left( t,x\right) dxdt+\displaystyle\int_{Q_{T}}\left\vert
v\left( t,x\right) \right\vert ^{p}\varphi \left( t,x\right) dxdt,
\label{formu1}
\end{eqnarray}%
\begin{eqnarray}
&&\displaystyle\int_{Q_{T}}vD_{t|T}^{\gamma _{2}}\varphi \left( t,x\right)
dxdt-\displaystyle\int_{Q_{T}}v\Delta \varphi \left( t,x\right) dxdt=%
\displaystyle\int_{\mathbb{R}^{N}}v_{0}\left( x\right) \left(
D_{t|T}^{\gamma _{2}-1}\varphi \right) \left( 0,.\right) dx  \notag \\
&&+\displaystyle\int_{Q_{T}}v_{1}\left( x\right) D_{t|T}^{\gamma
_{2}-1}\varphi \left( t,x\right) dxdt+\displaystyle\int_{Q_{T}}\left\vert
u\left( t,x\right) \right\vert ^{q}\varphi \left( t,x\right) dxdt.
\label{formut} \\
&&  \notag
\end{eqnarray}
On the other hand, we have from the definition of $\varphi $ that 
\begin{eqnarray}
&&\displaystyle\int_{Q_{T}}u\varphi _{1}\left( x\right) D_{t|T}^{\gamma
_{1}}\varphi _{2}\left( t\right) dxdt-\displaystyle\int_{Q_{T}}u\varphi
_{2}\left( t\right) \Delta \varphi _{1}\left( x\right) \, dxdt  \notag \\
&&=\displaystyle\int_{\mathbb{R}^{N}}u_{0}\left( x\right) \varphi _{1}\left(
x\right) \left( D_{t|T}^{\gamma _{1}-1}\varphi _{2}\right) \left( 0,.\right)
dx  \notag \\
&&+\displaystyle\int_{Q_{T}}u_{1}\varphi _{1}\left( x\right) D_{t|T}^{\gamma
_{1}-1}\varphi _{2}\left( t\right) dxdt+\displaystyle\int_{Q_{T}}\left\vert
v\left( t,x\right) \right\vert ^{p}\varphi _{1}\left( x\right) \varphi
_{2}\left( t\right) dxdt\text{,}
\end{eqnarray}%
and%
\begin{eqnarray}
&&\displaystyle\int_{Q_{T}}v\varphi _{1}\left( x\right) D_{t|T}^{\gamma
_{2}}\varphi _{2}\left( t\right) dxdt-\displaystyle\int_{Q_{T}}v\varphi
_{2}\left( t\right) \Delta \varphi _{1}\left( x\right) \, dxdt  \notag \\
&&=\displaystyle\int_{\mathbb{R}^{N}}v_{0}\left( x\right) \varphi _{1}\left(
x\right) \left( D_{t|T}^{\gamma _{2}-1}\varphi _{2}\right) \left( 0,.\right)
dx  \notag \\
&&+\displaystyle\int_{Q_{T}}v_{1}\left( x\right) \varphi _{1}\left( x\right)
D_{t|T}^{\gamma _{2}-1}\varphi _{2}\left( t\right) dxdt+\displaystyle%
\int_{Q_{T}}\left\vert u\left( t,x\right) \right\vert ^{q}\varphi _{1}\left(
x\right) \varphi _{2}\left( t\right) dxdt\text{.} \\
&&  \notag
\end{eqnarray}

Applying H\"{o}lder's inequality with exponents $q$ and $q^{\prime }=\frac{q%
}{q-1}$ to the right-hand side of (\ref{formu1}), we get 
\begin{eqnarray*}
\displaystyle\int_{Q_{T}}u\varphi _{1}\left( x\right) D_{t|T}^{\gamma
_{1}}\varphi _{2}\left( t\right) dxdt &=&\displaystyle\int_{Q_{T}}
u\left\vert \varphi _{2}\left( t\right) \right\vert ^{\frac{1}{q}}\left\vert
\varphi _{1}\left( x\right) \right\vert ^{1-\frac{1}{q}+\frac{1}{q}%
}\left\vert \varphi _{2}\left( t\right) \right\vert ^{-\frac{1}{q}%
}D_{t|T}^{\gamma _{1}}\varphi _{2}\left( t\right) dxdt \\
&\leq &\mathcal{I}^{\frac{1}{q}} \mathcal{\tilde A},
\end{eqnarray*}
where we have set 
\begin{equation*}
\mathcal{I} := \int_{Q_{T}}\left\vert u\right\vert ^{q}\varphi _{1}\left(
x\right) \varphi _{2} \, dxdt
\end{equation*}
\begin{equation*}
\mathcal{\tilde A}:= \left( \displaystyle \int_{Q_{T}}\left\vert
D_{t|T}^{\gamma _{1}}\varphi _{2}\left( t\right) \right\vert ^{q^{\prime
}}\left\vert \varphi _{2}\left( t\right) \right\vert ^{-\frac{q^{\prime }}{q}%
}\left\vert \varphi _{1}\left( x\right) \right\vert ^{\left( 1-\frac{1}{q}%
\right) q^{\prime }}dxdt\right) ^{\frac{1}{q^{\prime }}},
\end{equation*}
\begin{eqnarray*}
\displaystyle\int_{Q_{T}}u\Delta \varphi _{1}\left( x\right) \varphi
_{2}\left( t\right) \, dxdt &\leq &\mathcal{I}^{\frac{1}{q}}\left(
\int_{Q_{T}}\left\vert \Delta \varphi _{1}\left( x\right) \right\vert
^{q^{\prime }}\left\vert \varphi _{1}\left( x\right) \right\vert ^{-\frac{%
q^{\prime }}{q}}\left\vert \varphi _{2}\left( t\right) \right\vert ^{\left(
1-\frac{1}{q}\right) q^{\prime }}dxdt\right) ^{\frac{1}{q^{\prime }}} \\
&\leq &C\mathcal{I}^{\frac{1}{q}}\left( \int_{\text{supp}\left( \Delta
\varphi _{1}\right) }\varphi _{1}^{1-q^{\prime }}\left( x\right) \left\vert
\Delta \varphi _{1}\left( x\right) \right\vert ^{q^{\prime }}\left\vert
\varphi _{1}^{l}\left( x\right) \right\vert ^{-\frac{q^{\prime }}{q}%
}dx\int_{0}^{T}\left\vert \varphi _{2}\left( t\right) \right\vert ^{\left( 1-%
\frac{1}{q}\right) q^{\prime }} \, dt\right) ^{\frac{1}{q^{\prime }}}\text{.}
\end{eqnarray*}

Collecting the above estimates, we obtain 
\begin{eqnarray}
&&CT^{\left( 1-\gamma _{1}\right) }\displaystyle\int_{\mathbb{R}%
^{N}}u_{0}\left( x\right) \varphi _{1}\left( x\right) dx+CT^{2-\gamma _{1}}%
\displaystyle\int_{\mathbb{R}^{N}}u_{1}\left( x\right) \varphi _{1}\left(
x\right) dx+\mathcal{J}  \notag \\
&\leq &\mathcal{I} ^{\frac{1}{q}} \mathcal{\tilde A}  \notag \\
&&+\mathcal{I} ^{\frac{1}{q}}\left( \int_{\mathbb{R}^{N}}\left\vert \Delta
\varphi _{1}\left( x\right) \right\vert ^{q^{\prime }}\left\vert \varphi
_{1}\left( x\right) \right\vert ^{-\frac{q^{\prime }}{q}}dx\int_{0}^{T}\left%
\vert \varphi _{2}\left( t\right) \right\vert ^{\left( 1-\frac{1}{q}\right)
q^{\prime }}dt\right) ^{\frac{1}{q^{\prime }}},  \label{in1}
\end{eqnarray}
where we have set 
\begin{equation*}
\mathcal{J}:=\int_{Q_{T}}\left\vert v\right\vert ^{p}\varphi _{1}\left(
x\right) \varphi _{2}\left( t\right) dxdt.
\end{equation*}%
Similarly, we obtain 
\begin{eqnarray}
&&\mathcal{I}+T^{\left( 1-\gamma _{2}\right) }\int_{\mathbb{R}%
^{N}}v_{0}\varphi _{1}\left( x\right) dx + CT^{2-\gamma _{2}}\displaystyle%
\int_{\mathbb{R}^{N}}v_{1}\left( x\right) \varphi _{1}\left( x\right) dx 
\notag \\
&\leq &\mathcal{J} ^{\frac{1}{p}}\left( \displaystyle\int_{Q_{T}}\left\vert
D_{t|T}^{\gamma _{2}}\varphi _{2}\left( t\right) \right\vert ^{p^{\prime
}}\left\vert \varphi _{2}\left( t\right) \right\vert ^{-\frac{p^{\prime }}{p}%
}\left\vert \varphi _{1}\left( x\right) \right\vert ^{\left( 1-\frac{1}{p}%
\right) p^{\prime }}dxdt\right) ^{\frac{1}{p^{\prime }}}  \notag \\
&&+\mathcal{J} ^{\frac{1}{p}}\left( \int_{\mathbb{R}^{N}}\left\vert \Delta
\varphi _{1}\left( x\right) \right\vert ^{p^{\prime }}\left\vert \varphi
_{1}\left( x\right) \right\vert ^{-\frac{p^{\prime }}{p}}dx\int_{0}^{T}\left%
\vert \varphi _{2}\left( t\right) \right\vert ^{\left( 1-\frac{1}{p}\right)
p^{\prime }}dt\right) ^{\frac{1}{p^{\prime }}},  \label{in2}
\end{eqnarray}
where $pp^{\prime }=p+p^{\prime }$.\newline

Consequently, 
\begin{eqnarray*}
\mathcal{J}+CT^{\left( 1-\gamma _{1}\right) }\displaystyle\int_{\mathbb{R}%
^{N}}u_{0}\varphi _{1}^{l}\left( x\right) dx+CT^{2-\gamma _{1}}\displaystyle%
\int_{\mathbb{R}^{N}}u_{1}\left( x\right) \varphi _{1}\left( x\right) \, dx
\leq \mathcal{A}\mathcal{I}^{\frac{1}{q}},
\end{eqnarray*}
and 
\begin{eqnarray*}
\mathcal{I}+CT^{\left( 1-\gamma _{2}\right) }\int_{\mathbb{R}^{N}}v_{0}\Phi
\left( x\right) dx+CT^{2-\gamma _{2}}\displaystyle\int_{\mathbb{R}%
^{N}}v_{1}\left( x\right) \varphi _{1}\left( x\right) \, dx \leq \mathcal{B}%
\mathcal{J} ^{\frac{1}{p}},
\end{eqnarray*}
with 
\begin{eqnarray*}
\mathcal{A} &=&\left( \displaystyle\int_{Q_{T}}\left\vert D_{t|T}^{\gamma
_{1}}\varphi _{2}\left( t\right) \right\vert ^{q^{\prime }}\left\vert
\varphi _{2}\left( t\right) \right\vert ^{-\frac{q^{\prime }}{q}}\left\vert
\varphi _{1}\left( x\right) \right\vert ^{\left( 1-\frac{1}{q}\right)
q^{\prime }}dxdt\right) ^{\frac{1}{q^{\prime }}} \\
&&+\left( \int_{Q_{T}}\left\vert \Delta \varphi _{1}\left( x\right)
\right\vert ^{q^{\prime }}\left\vert \varphi _{1}\left( x\right) \right\vert
^{-\frac{q^{\prime }}{q}}\left\vert \varphi _{2}\left( t\right) \right\vert
^{\left( 1-\frac{1}{q}\right) q^{\prime }}dxdt\right) ^{\frac{1}{q^{\prime }}%
} \\
&\leq &CT^{\left( -q^{\prime }\gamma _{1}+1+N\lambda \right) \frac{1}{%
q^{\prime }}}+CT^{\left( -2\lambda q^{\prime }+1+N\lambda \right) \frac{1}{%
q^{\prime }}}\text{,}
\end{eqnarray*}
and 
\begin{eqnarray*}
\mathcal{B} &=&\left( \displaystyle\int_{Q_{T}}\left\vert D_{t|T}^{\gamma
_{2}}\varphi _{2}\left( t\right) \right\vert ^{p^{\prime }}\left\vert
\varphi _{2}\left( t\right) \right\vert ^{-\frac{p^{\prime }}{p}}\left\vert
\varphi _{1}\left( x\right) \right\vert ^{\left( 1-\frac{1}{p}\right)
p^{\prime }}dxdt\right) ^{\frac{1}{p^{\prime }}} \\
&&+\left( \displaystyle\int_{Q_{T}}\left\vert \Delta \varphi _{1}\left(
x\right) \right\vert ^{p^{\prime }}\left\vert \varphi _{1}^{l}\left(
x\right) \right\vert ^{-\frac{p^{\prime }}{p}}\left\vert \varphi _{2}\left(
t\right) \right\vert ^{\left( 1-\frac{1}{p}\right) p^{\prime }}dxdt\right) ^{%
\frac{1}{p^{\prime }}} \\
&\leq &CT^{\left( -\gamma _{2}p^{\prime }+1+N\lambda \right) \frac{1}{%
p^{\prime }}}+CT^{\left( -2\lambda p^{\prime }+1+N\lambda \right) \frac{1}{%
p^{\prime }}}.
\end{eqnarray*}
Using inequalities (\ref{in1}) and (\ref{in2}), we can write 
\begin{equation*}
\mathcal{J}+CT^{2-\gamma _{1}}\displaystyle\int_{\mathbb{R}^{N}}u_{1}\left(
x\right) \varphi _{1}\left( x\right) dx\leq \mathcal{A} \, \mathcal{B}^{%
\frac{1}{q}}\mathcal{J} ^{\frac{1}{pq}},
\end{equation*}
and 
\begin{eqnarray*}
\mathcal{I}+CT^{2-\gamma _{2}}\displaystyle\int_{\mathbb{R}^{N}}v_{1}\left(
x\right) \varphi _{1}^{l}\left( x\right) \, dx \leq \mathcal{B}\, \mathcal{A}%
^{\frac{1}{p}}\mathcal{I} ^{\frac{1}{pq}}.
\end{eqnarray*}
Now, applying Young's inequality to the right hand side of the above
estimates, we get 
\begin{equation*}
\left( pq-1\right) \mathcal{J}+CpqT^{2-\gamma _{1}}\displaystyle\int_{%
\mathbb{R}^{N}}u_{1}\left( x\right) \varphi _{1}\left( x\right) \, dx \leq
\left( pq-1\right) \left( \mathcal{A}\,\mathcal{B}^{\frac{1}{q}}\right) ^{%
\frac{pq}{pq-1}},
\end{equation*}
and 
\begin{equation*}
\left( pq -1\right) \mathcal{I}+Cpq T^{2-\gamma _{2}}\displaystyle\int_{%
\mathbb{R}^{N}}v_{1}\left( x\right) \varphi _{1}\left( x\right) dx \newline
\leq \left( pq-1\right) \left( \mathcal{B}\,\mathcal{A}^{\frac{1}{p}}\right)
^{\frac{pq}{pq-1}}.  \notag
\end{equation*}
At this stage, we set $x=T^{\lambda }y,t=T\tau ,$ with $\lambda >0$ to be
chosen later. Then we have 
\begin{equation*}
\mathcal{A} \, \mathcal{B}^{\frac{1}{q}}\leq CT^{\left[ \left( -q^{\prime
}\gamma _{1}+1+\lambda N\right) \frac{1}{q^{\prime }}+\left( -p^{\prime
}\gamma _{2}+1+N\lambda \right) \frac{1}{qp^{\prime }}\right] \frac{pq}{pq-1}%
}+T^{\left[ \left( -q^{\prime }\gamma _{1}+1+\lambda N\right) \frac{1}{%
q^{\prime }}+\left( -2\lambda p^{\prime }+1+N\lambda \right) \frac{1}{%
qp^{\prime }}\right] \frac{pq}{pq-1}}\text{,}
\end{equation*}
and 
\begin{equation*}
\mathcal{B}\, \mathcal{A}^{\frac{1}{p}}\leq C\left( T^{\left( -\gamma
_{2}p^{\prime }+1+N\lambda \right) \frac{1}{p^{\prime }}}+T^{\left(
-2\lambda p^{\prime }+1+N\lambda \right) \frac{1}{p^{\prime }}}\right)
\left( T^{\left( -q^{\prime }\gamma _{1}+1+N\lambda \right) \frac{1}{%
q^{\prime }}}+T^{\left( -2\lambda q^{\prime }+1+N\lambda \right) \frac{1}{%
q^{\prime }}}\right) ^{\frac{1}{p}}\text{.}
\end{equation*}
We choose $\lambda =\frac{\gamma _{1}}{2}$ so that ($\left( -q^{\prime
}\gamma _{1}+1+N\lambda \right) \frac{1}{q^{\prime }}=\left( -2\lambda
q^{\prime }+1+N\lambda \right) \frac{1}{q^{\prime }})$.

Therefore, we have%
\begin{equation}
\displaystyle\int_{\mathbb{R}^{N}}u_{1}\left( x\right) \varphi _{1}\left(
x\right) dx\leq CT^{\delta _{1}}\text{,}  \label{cigm1}
\end{equation}%
and%
\begin{equation}
\int_{\mathbb{R}^{N}}v_{1}\left( x\right) \varphi _{1}\left( x\right) dx\leq
CT^{\delta _{2}}\text{,}  \label{cigm2}
\end{equation}%
where 
\begin{eqnarray*}
&&\delta _{1}=\max \left\{ \left[ \left( -q^{\prime }\gamma _{1}+1+\frac{%
\gamma _{1}}{2}N\right) \frac{1}{q^{\prime }}+\left( -p^{\prime }\gamma
_{2}+1+N\frac{\gamma _{1}}{2}\right) \frac{1}{qp^{\prime }}\right] \frac{pq}{%
pq-1}+\gamma _{1}-2,\right. \\
&&\left. \left[ \left( -q^{\prime }\gamma _{1}+1+\frac{\gamma _{1}}{2}%
N\right) \frac{1}{q^{\prime }}+\left( -2\frac{\gamma _{1}}{2}p^{\prime }+1+N%
\frac{\gamma _{1}}{2}\right) \frac{1}{qp^{\prime }}\right] \frac{pq}{pq-1}%
+\gamma _{1}-2\right\} ,
\end{eqnarray*}%
and 
\begin{eqnarray*}
&&\delta _{2}=\max \left\{ \left[ \left( -\gamma _{2}p^{\prime }+1+N\frac{%
\gamma _{1}}{2}\right) \frac{1}{p^{\prime }}+\left( -q^{\prime }\gamma
_{1}+1+N\frac{\gamma _{1}}{2}\right) \frac{1}{pq^{\prime }}\right] \frac{pq}{%
pq-1}+\gamma _{2}-2,\right. \\
&&\left. \left[ \left( -2\frac{\gamma _{1}}{2}p^{\prime }+1+N\frac{\gamma
_{1}}{2}\right) \frac{1}{p^{\prime }}+\left( -q^{\prime }\gamma _{1}+1+N%
\frac{\gamma _{1}}{2}\right) \frac{1}{pq^{\prime }}\right] \frac{pq}{pq-1}%
+\gamma _{2}-2\right\} .
\end{eqnarray*}%
The condition (\ref{critdimension}) leads to either $\delta _{1}<0$ or $%
\delta _{2}<0$. Then, as $T\rightarrow \infty $, the right-hand side of (\ref%
{cigm1})(resp. (\ref{cigm2})) tends to zero and the left-hand side converges
to $\displaystyle\int_{\mathbb{R}^{N}}u_{1}\left( x\right) dx>0$ (resp. $%
\int_{\mathbb{R}^{N}}v_{1}\left( x\right) dx>0)$, which is contradiction.

We repeat the same argument with $\lambda =\frac{\gamma _{2}}{2}$ to
conclude the proof of the Theorem\ref{NEG}.
\end{proof}

\begin{remark}
\textrm{When }$\gamma _{1}=\gamma _{2}=\gamma $\textrm{, we recover the case
studied by }\cite{AlmeidaEJDE}.
\end{remark}


\begin{flushleft}
\textbf{Ahmad Bashir, Ahmed Alsaedi}\newline
NAAM Research Group, Department of Mathematics,\newline
Faculty of Science, King Abdulaziz University, P.O. Box 80203 Jeddah 21589,%
\newline
Saudi Arabia.\newline

\textbf{Mohamed Berbiche}\newline
Laboratory of Mathematical Analysis, Probability and Optimizations,\newline
Mohamed Khider University, Biskra,Po. Box 145 Biskra ( 07000),\newline
Algeria.\newline

\textbf{Mokhtar Kirane}\newline
NAAM Research Group, Department of Mathematics,\newline
Faculty of Science, King Abdulaziz University, P.O. Box 80203 Jeddah 21589,%
\newline
Saudi Arabia, and\newline
LaSIE, Universit\'{e} de La Rochelle, P\^{o}le Sciences et Technologies,%
\newline
Avenue Michel Cr\'{e}peau, 17000 La Rochelle,\newline
France.
\end{flushleft}

\end{document}